\documentclass[11pt]{article}
\usepackage{amsmath,amsfonts,amssymb, amsthm,enumerate,graphicx}
\usepackage{tikz,authblk}
\usepackage{float}
\usepackage{subfig}
\usepackage{url}
\newtheorem{theorem} {{\textsf{Theorem}}}
\newtheorem{proposition}[theorem]{{\textsf{Proposition}}}
\newtheorem{corollary}[theorem]{{\textsf{Corollary}}}

\newtheorem{definition}[theorem]{{\textsf{Definition}}}
\newtheorem{remark}[theorem]{{\textsf{Remark}}}

\newtheorem{lemma}[theorem]{{\textsf{Lemma}}}

\newcommand{\TPSS}{\mathbb{S}^{\hspace{.2mm}2} \mbox{$\times
\hspace{-2.6mm}_{-}$} \, \mathbb{S}^{\hspace{.1mm}1}}

\begin{document}

\title{Lower bounds for regular genus and gem-complexity of PL 4-manifolds with boundary}
\author{Biplab Basak and Manisha Binjola}

\date{}

\maketitle

\vspace{-15mm}
\begin{center}

\noindent {\small Department of Mathematics, Indian Institute of Technology Delhi, New Delhi 110016, India.$^1$}


\footnotetext[1]{{\em E-mail addresses:} \url{biplab@iitd.ac.in} (B.
Basak), \url{binjolamanisha@gmail.com} (M. Binjola).}

\medskip

\date{October 28, 2020}
\end{center}

\hrule

\begin{abstract}
Let $M$ be a connected compact PL 4-manifold with boundary. In this article, we have given several lower bounds for regular genus and gem-complexity of the manifold $M$. In particular, we have proved that if $M$ is a connected compact $4$-manifold with $h$ boundary components then its gem-complexity $\mathit{k}(M)$ satisfies the following inequalities:
 $$\mathit{k}(M)\geq 3\chi(M)+7m+7h-10  \mbox{ and }\mathit{k}(M)\geq \mathit{k}(\partial M)+3\chi(M)+4m+6h-9,$$
 and its regular genus $\mathcal{G}(M)$ satisfies the following inequalities:
 $$\mathcal{G}(M)\geq 2\chi(M)+3m+2h-4\mbox{ and }\mathcal{G}(M)\geq \mathcal{G}(\partial M)+2\chi(M)+2m+2h-4,$$
 where $m$ is the rank of the fundamental group of the manifold $M$. These lower bounds enable to strictly improve previously known estimations for regular genus and gem-complexity of a PL $4$-manifold with boundary. Further, the sharpness of these bounds has also been shown for a large class of PL $4$-manifolds with boundary. 
 \end{abstract}

\noindent {\small {\em MSC 2020\,:} Primary 57Q15; Secondary 57Q05, 57K41, 05C15.

}

\medskip

\section{Introduction}

A crystallization $(\Gamma,\gamma)$ (cf. Definition \ref{def:crystallization}) of a connected compact PL $d$-manifold (possibly with boundary) is a certain type of edge colored graph which represents the manifold (for details and related notations we refer Subsection \ref{crystal}).  The existence of a crystallization  for every connected closed PL $d$-manifold follows from a result by Pezzana (see \cite{pe74}), and later crystallization is defined and its existence has been proved for every PL $d$-manifold with boundary (see \cite{cg80, ga83}).  The gem-complexity  of a connected compact PL $d$-manifolds $M$ (possibly with boundary) is defined as $\mathit{k}(M) = p - 1$, where $2p$ is the least number of vertices of a crystallization of $M$ (cf. Definition \ref{def:gem-complexity}). In \cite{bd14}, the authors gave a lower bound for the gem-complexity of a closed connected 3-manifold in terms of the weight of its fundamental group. In \cite{bb20}, we gave a lower bound for the gem-complexity of a connected compact 3-manifold $M$ with boundary $\partial M$ in terms of the genus of $\partial M$ and number of the boundary components of $\partial M$. In \cite{bc15}, the authors gave a lower bound for the gem-complexity of a closed connected PL 4-manifold in terms of the rank of the fundamental group of the $4$-manifold. Catalogues for lower dimensional closed connected PL manifolds via gem-complexity can be found in \cite{cc08, cc15}. The results on gem-complexity for connected compact PL 4-manifold with boundary are not well known in the literature. In this article, we have provided a nice property for the number of vertices of a crystallization of  a  connected compact PL 4-manifold with boundary (cf. Lemmas \ref{lemma:2pand2pbar} and \ref{lemma:bound}), and using this  we have given a lower bound for the gem-complexity of a connected compact PL 4-manifold $M$ with $h$ boundary components:  $\mathit{k}(M)\geq 3\chi(M)+7m+7h-10$ and $\mathit{k}(M)\geq \mathit{k}(\partial M)+3\chi(M)+4m+6h-9$, where $m$ is the rank of the fundamental group of the manifold $M$ (cf. Theorem \ref{theorem:main1}). We have also shown that these bounds are sharp for a large class of connected compact PL 4-manifolds with boundary (cf. Remark \ref{remark:gem-complexity}).

Extending the concept of genus in dimension 2, the concept of regular genus for a closed connected PL $d$-manifold, has been introduced in \cite{ga81}, which is related to the existence of regular embeddings of crystallizations of the manifold into surfaces (cf. Subsection \ref{sec:genus} for details). Later, in \cite{ga87}, the concept of regular genus has been extended for a connected compact PL $d$-manifold with boundary, for $d\geq 2$ (cf. Definition \ref{def:regular-genus}). Note that the notion of regular genus extends the classical notion of genus to arbitrary dimension: in fact, the regular genus of a closed connected orientable (resp. non-orientable) surface is same as its genus (resp. half of its genus), while the regular genus of a closed connected 3-manifold equals its Heegaard genus (see  \cite{ga81,[GG]}). It is known that, the regular genus zero characterizes the $d$-sphere among all closed connected PL $d$-manifolds (\cite{fg82}). The concept of regular genus is a very important and useful combinatorial invariant in topology. The additivity of regular genus under connected sum has been conjectured, and the associated (open) problem is significant, especially in dimension four. In fact, the additivity of regular genus for closed simply-connected PL $4$-manifolds would imply the 4-dimensional Smooth  Poincar\'{e} Conjecture (cf. \cite{cav99}). Further, a lot of theory has been established to find bounds for regular genus. Obviously, any crystallization of a given $d$-manifold $M$ yields an upper bound for the regular genus of $M$ but the problem of finding lower bounds is generally more difficult. In \cite{cav99}, the author gave a lower bound for the regular genus of a closed connected PL $4$-manifold in terms of the second Betti number of the manifold. In \cite{bc15}, the authors gave a lower bound for the regular genus of a closed connected PL $4$-manifold in terms of the Euler characteristic and the  rank of the fundamental group of the manifold. In \cite{cp90}, the authors gave two lower bounds for the regular genus of a connected compact PL $d$-manifold with boundary. They proved that for a connected compact PL $d$-manifold $M$ with boundary $\partial M$, the regular genus $\mathcal{G}(M) \geq \mathcal{G}(\partial M)$ and $\mathcal{G}(M) \geq m$,  where $m$ is the rank of the fundamental group of the manifold $M$. From \cite{ca92,cm97}, it is known that, in dimension 4, $\mathcal{G}(M) = \mathcal{G}(\partial M)$ if and only if $M$ is a connected sum of handlebodies. Thus, if $M$ is not a connected sum of handlebodies then the bound $\mathcal{G}(M) \geq \mathcal{G}(\partial M)$ will never be sharp. Recently  in \cite{cc20}, the authors studied the combinatorial properties of compact PL $4$-manifolds with boundary in a different way. For a compact $4$-manifold with  at most one non-spherical boundary component, first they constructed a singular manifold by coning off its boundary component, and then they gave a lower bound for the regular genus of the singular manifold. A lower bound problem  for regular genus of PL $d$-manifolds with boundary has a significant role in combinatorial topology because a lower bound for regular genus of PL $d$-manifolds with boundary allows a new classification problem via the `restricted gap' between the regular genus of the manifold and the obtained lower bound. For $3\leq d \leq 5$, a lot of classifying results in PL-category have been obtained for both closed connected PL $d$-manifolds and connected compact PL $d$-manifolds with boundary (cf. \cite{bb19,bb18, [C], Casali, [CG], cav99}). These results concern the case of `low' regular genus, the case of `restricted gap' between the regular genus of the manifold and the regular genus of its boundary, and the case of `restricted gap' between the regular genus and the rank of the fundamental group of the manifold. These show the importance of  lower bound problems for regular genus  in combinatorial topology, and motivate us for studying  lower bound problems for regular genus of PL $4$-manifolds with several boundary components. In this article,  we have given  lower bounds for the regular genus of a connected compact PL 4-manifold $M$ with $h$ boundary components:  $\mathcal{G}(M)\geq 2\chi(M)+3m+2h-4$ and $\mathcal{G}(M)\geq \mathcal{G}(\partial M)+2\chi(M)+2m+2h-4$ (cf. Theorem \ref{theorem:main2}).  These bounds enable to strictly improve previously known estimations for the regular genus of any connected compact PL $4$-manifolds with boundary. These bounds follow from our main result Lemma \ref{lemma:main}. In fact, Lemma \ref{lemma:main} is the heart of this article, which will be very useful to classify the connected compact PL 4-manifolds $M$ with boundary $\partial M$ via regular genera of $M$ and $\partial M$. Further, we have proved that, our bounds for regular genus are sharp for a large class of PL $4$-manifolds with boundary  (cf. Remark \ref{remark:regular-genus}).

\section{Preliminaries}

\subsection{Crystallization} \label{crystal}

Crystallization theory is a combinatorial tool for the representaion of piecewise-linear (PL) manifolds of arbitrary dimension. A multigraph is a graph where multiple edges are allowed but loops are restricted. For a multigraph $\Gamma= (V(\Gamma),E(\Gamma))$, a surjective map $\gamma : E(\Gamma) \to \Delta_d:=\{0,1, \dots , d\}$ is called a proper edge-coloring if $\gamma(e_1) \ne \gamma(e_2)$ for any pair $e_1,e_2$ of adjacent edges. The members of the set $\Delta_d$ are called the {\it colors} of $\Gamma$. A graph $(\Gamma,\gamma)$ is called {\it $(d+1)$-regular} if degree of each vertex is $d+1$ and is said to be {\it $(d+1)$-regular with respect to a color $c$} if the graph is $d$-regular after removing all the $c$- colored edges. We refer to \cite{bm08} for the standard terminology on graphs.

A graph $(\Gamma,\gamma)$ is called {\it $(d+1)$-regular colored graph} if $\Gamma$ is a $(d+1)$-regular and $\gamma$ is a proper edge-coloring.  A {\it $(d+1)$-colored graph with boundary} is a pair $(\Gamma,\gamma)$, where $\Gamma$ is a $(d+1)$-regular graph with respect to a color $c\in \Delta_d$ but not $(d+1)$-regular and $\gamma$ is a proper edge-coloring. If $(\Gamma,\gamma)$ is a $(d+1)$-regular colored graph or a $(d+1)$-colored graph with boundary then we simply call $(\Gamma,\gamma)$ as a $(d+1)$-colored graph. For given $B \subseteq \Delta_d$ with cardinality $k$ , the graph $\Gamma_B =(V(\Gamma), \gamma^{-1}(B))$ is a $k$-colored graph with edge-coloring $\gamma|_{\gamma^{-1}(B)}$. For a color set $\{i_1,i_2,\dots,i_k\} \subset \Delta_d$, $\Gamma_{\{i_1,i_2, \dots, i_k\}}$ denotes the subgraph of $\Gamma$ restricted to the color set  $\{i_1,i_2,\dots,i_k\}$  and $g_{i_1i_2 \cdots i_k}$ (or sometimes $g(\Gamma_{\{i_1,i_2, \dots, i_k\}})$) denotes the number of connected components of the subgraph $\Gamma_{\{i_1, i_2, \dots, i_k\}}$. Let 
$\dot g_{i_1i_2 \cdots i_k}$ denote the number of $k$-regular connected components of the graph $\Gamma_{\{i_1, i_2, \dots, i_k\}}$, i.e., $\dot g_{ij}$ denotes the number of $\{i,j\}$-colored cycles in $\Gamma$.  Let $2\overline{p}$ and $2\dot{p}$ denote the number of boundary vertices and internal vertices respectively. Then $2p=2\overline{p}+2\dot{p}$ is the total number of vertices of $(\Gamma,\gamma)$. A graph $(\Gamma,\gamma)$ is called {\it contracted} if subgraph $\Gamma_{\hat{c}}:=\Gamma_{\Delta_d\setminus \{c\}}$ is connected for all $c$.

Let $\mathbb{G}_d$ denote the set of all graphs $(\Gamma,\gamma)$ which are $(d+1)$-regular with respect to the fixed color $d$. In particular, $\mathbb{G}_d$ contains all $(d+1)$-regular colored graphs. If $(\Gamma,\gamma)\in \mathbb{G}_d$ then the vertex with degree $d+1$ is called internal vertex and the vertex with degree $d$ is called boundary vertex.  A boundary graph $(\partial\Gamma,\partial \gamma)$ for the graph $(\Gamma,\gamma) \in \mathbb{G}_d $ is defined as follows:

\begin{itemize}
\item{} there is a bijection between $V(\partial \Gamma)$ and the set of boundary vertices of $\Gamma$;

\item{} for $0\leq j \leq d-1$, vertices $u_1$ and $u_2$ are joined by $j$-colored edge in $\partial \Gamma$ if and only if corresponding vertices in $\Gamma$ are joined by a path formed by $j$ and $d$-colored edges alternatively.
\end{itemize}

Note that, if $(\Gamma,\gamma)$ is $(d+1)$-regular then $\partial \Gamma = \emptyset$ and $(\Gamma,\gamma)\in \mathbb{G}_d$. For each  $(\Gamma,\gamma) \in \mathbb{G}_d$, a corresponding $d$-dimensional simplicial cell-complex ${\mathcal K}(\Gamma)$ is determined as follows:

\begin{itemize}
\item{} corresponding to each vertex $u\in V(\Gamma)$, take a $d$-simplex $\sigma(u)$ with vertices labeled by $\Delta_d$;

\item{} for each $j$-colored edge between $u,v\in V(\Gamma)$, identify ($d-1$)-faces of simplices $\sigma(u)$ and $\sigma(v)$ opposite to $j$-labeled vertices such that the vertices with same label identify with each other.
\end{itemize}

We refer to \cite{bj84} for CW-complexes and related terminologies. We say  $(\Gamma,\gamma)$ {\it represents} a PL $d$-manifold $M$ (possibly with boundary) if the geometrical carrier $|{\mathcal K}(\Gamma)|$ is PL homeomorphic to $M$. It can be easily seen that  $|{\mathcal K}(\Gamma)|$ is orientable if and only if the graph $\Gamma$ is bipartite. If  $(\Gamma,\gamma)\in \mathbb{G}_d$ represents a PL $d$-manifold with boundary then each component of its boundary-graph $(\partial \Gamma,\partial \gamma)$ represents a component of $\partial M$.

 Let $(\Gamma,\gamma)\in \mathbb{G}_d$ represent a connected compact PL $d$-manifold with $h$ boundary  components ($h\geq1$). Then  $\mathcal{K}(\Gamma)$ has at least $dh+1$ vertices. It is easy to see that $\Gamma_{\hat d}$ is connected and each component of $\partial \Gamma$ is contracted if and only if $\mathcal{K}(\Gamma)$ has exactly $dh+1$ vertices.  If $(\Gamma,\gamma)$ represents a connected closed PL $d$-manifold then $\mathcal{K}(\Gamma)$ has at least $d+1$ vertices.
\begin{definition}[\cite{ga83}]\label{def:crystallization}
 A contracted graph $(\Gamma,\gamma)$ is called a crystallization of a connected closed PL $d$-manifold $M$ if  $(\Gamma,\gamma)$ represents $M$ (i.e., $\mathcal{K}(\Gamma)$ has exactly  $d+1$ vertices).
A connected graph  $(\Gamma,\gamma) \in \mathbb{G}_d$ is called a crystallization of a connected compact PL $d$-manifold $M$ with $h$ boundary components ($h\geq 1$) if   $(\Gamma,\gamma)$ represents $M$  and $\mathcal{K}(\Gamma)$ has exactly  $dh+1$ vertices.
\end{definition}

Thus, if $(\Gamma,\gamma) \in \mathbb{G}_d$ is a crystallization of a connected compact PL $d$-manifold $M$ with $h$ boundary components ($h\geq 1$) then  $g(\Gamma_{\hat{d}})=1$ and $g(\Gamma_{\hat{c}})=h$ for $0\leq c \leq d-1$. In other words, $\mathcal K(\Gamma)$ has exactly $h$ number of vertices labeled by the color $c$, and each boundary component contains exactly one vertex labeled by the color $c$,  for $0\leq c \leq d-1$.
Note that, $(\Gamma,\gamma)$ is contracted if and only if $h= 1$.

The starting point of the whole crystallization theory is the following Pezzana's Existence Theorem (cf. \cite{pe74}).

\begin{proposition}
 Every closed connected PL $n$-manifold admits a crystallization.
\end{proposition}

Pezzana's existence theorem has been extended to the boundary case (cf \cite{cg80, ga83}). Let $(\Gamma,\gamma)$ and $(\Gamma',\gamma')$   be two $(d+1)$-colored graphs with color sets $\Delta_d$ and $\Delta'_d$ respectively. Then $I(\Gamma):=(I_V,I_c):\Gamma \to \Gamma'$ is said to be an {\em isomorphism} if $I_V: V(\Gamma) \to V(\Gamma')$ and $I_c:\Delta_d \to \Delta'_d$ are bijective maps such that $u_1 u_2$ is an edge of color $j \in \Delta_d$ if and only if $I_V(u_1)I_V(u_2)$ is an edge of color $I_c(j) \in \Delta'_d$. And  $(\Gamma,\gamma)$ and $(\Gamma',\gamma')$ are said to be isomorphic.

\begin{proposition}[\cite{cg80, ga83}]
Let $M$ be a connected compact $n$-manifold with (possibly non connected) boundary. For every crystallization $(\Gamma',\gamma')$ of $\partial M$, there exists a crystallization $(\Gamma,\gamma)$ of $M$, whose boundary graph $(\partial \Gamma,\partial \gamma)$ is isomorphic with $(\Gamma',\gamma')$.
\end{proposition}

It is known that for a given PL $d$-manifold $M$ with boundary, there exists a $(d+1)$-colored graph $(\Gamma,\gamma)$ which is  regular with respect to a fixed color $k\in \Delta_{d}$ and represents $M$. Without loss of generality, we assume $k=d$, i.e., $(\Gamma,\gamma) \in \mathbb{G}_d$.  A new invariant `gem-complexity' has been defined. 
\begin{definition}\label{def:gem-complexity}
The  gem-complexity of a connected compact PL $d$-manifold $M$ (possibly with boundary) is the non-negative integer $\mathit{k}(M) = p - 1$, where $2p$ is the least number of vertices of a crystallization of $M$. 
\end{definition}

From the construction it is obvious that for $\mathcal{B} \subset \Delta_d$ of cardinality $k+1$, ${\mathcal K}(\Gamma)$ has as many $k$-simplices with vertices labeled by $\mathcal{B}$ as many connected components of $\Gamma_{\Delta_d \setminus \mathcal{B}}$ are (cf. \cite{fgg86}). Moreover, if  $\mathcal{B} \subseteq \Delta_d$ then each regular connected component of $\Gamma_{\Delta_d \setminus \mathcal{B}}$ represents an interior $k$-simplex with vertices labeled by $\mathcal{B}$ in  ${\mathcal K}(\Gamma)$ whereas non-regular represents a boundary $h$-simplex with vertices labeled by $\mathcal{B}$ in  ${\mathcal K}(\Gamma)$. 

Given two $(d+1)$-colored graphs $(\Gamma_1,\gamma_1)$ and $(\Gamma_2,\gamma_2)$ with the same color set $\Delta_d$ and two vertices $v_1 \in V(\Gamma_1), v_2 \in V(\Gamma_2)$, the graph connected sum $(\Gamma_1\#_{v_1v_2}\Gamma_2,\gamma_1\#_{v_1v_2}\gamma_2)$ is defined by deleting $v_1$ and $v_2$ from $(\Gamma_1,\gamma_1)$  and $(\Gamma_2,\gamma_2)$ respectively, and pasting together the pairs of `free' edges (the ones which had an end-point in the deleted vertices) with the same color. From the construction it is clear that $\mathcal K (\Gamma_1\#_{v_1v_2}\Gamma_2)$ is obtained from $\mathcal K (\Gamma_1)$ and $\mathcal K (\Gamma_2)$ by removing the $d$-simplices $\sigma(v_1)$ and $\sigma(v_2)$ and pasting together all $(\Delta_d \setminus \{i\})$-colored free $(d-1)$-simplices (the ones which were attached with $\sigma(v_1)$ and $\sigma(v_2)$) of $\mathcal K (\Gamma_1)-\sigma(v_1)$ and $\mathcal K (\Gamma_2)-\sigma(v_2)$, for $0\leq i \leq d$. For $0\leq i \leq d$, if the $i$-colored vertex of $\sigma(v_1)$ or $\sigma(v_2)$ is not a boundary vertex of $\mathcal K (\Gamma_1)$ or $\mathcal K (\Gamma_2)$ respectively then $(\Gamma_1\#_{v_1v_2}\Gamma_2,\gamma_1\#_{v_1v_2}\gamma_2)$ represents $|\mathcal K (\Gamma_1)| \# |\mathcal K (\Gamma_2)|$.  Here we list few well known results which will be needful for this article. In \cite{ga79}, Gagliardi gave the following characterization of a $4$-colored graph to be a crystallization of a closed connected 3-manifold.

\begin{proposition}\label{prop:3mfd}
Let $(\Gamma,\gamma)$ be a contracted $4$-colored graph with $n$ number of vertices and color set $\Delta_3=\{0,1,2,3\}$.  Then $(\Gamma,\gamma)$ is a crystallization of a closed connected $3$-manifold if and only if

\begin{enumerate}[$(i)$]
\item $g_{ij}=g_{kl}$ for $i \neq j \neq k \neq l\in \{0,1,2,3\}$, and
\item $g_{01}+g_{02}+g_{03}=2+n/2$.
\end{enumerate}
\end{proposition}

Moreover, we have the following similar necessary criteria for any closed connected PL $d$-manifold (cf. \cite{fgg86}).

\begin{proposition} \label{prop:preliminaries}
Let $(\Gamma,\gamma)$ be a $(d+1)$-colored graph which represents a closed connected PL $d$-manifold $M$, with $d \geq 3$. Then,   $2g_{ijk}=g_{ij}+g_{ik}+g_{jk}-\frac{\#V(\Gamma)}{2}$ for any distinct $i,j,k \in \Delta_d$.
\end{proposition}

From \cite{cp90,chi86}, we know the following.

\begin{proposition}\label{prop:rank}
Let $(\Gamma,\gamma)\in \mathbb{G}_4$ be a crystallization of a connected compact PL $4$-manifold $M$ (possibly with boundary). Then, for $a,b \in \Delta_4$, $$m \leq g(\Gamma_{\Delta_4\setminus\{a,b\}})-\big(g(\Gamma_{\Delta_4\setminus\{a\}})+g(\Gamma_{\Delta_4\setminus\{b\}})-1\big),$$ where $m$ is the rank of the fundamental group of $M$.
\end{proposition}

\subsection{Regular Genus of PL $d$-manifolds (possibly with boundary)}\label{sec:genus}

Let $(\Gamma,\gamma) \in \mathbb{G}_d$ be a $(d+1)$-colored graph which represents a connected compact $d$-manifold $M$ (possibly with boundary $\partial M$).  Let $2\overline{p}$ and $2\dot{p}$ denote the number of boundary vertices and internal vertices respectively. For each $u \in V(\partial \Gamma)$ (possibly empty), take a new vertex $\tilde u$. Join $u$ and $\tilde u$ by a new $d$-colored edge and thus a new graph $(\tilde \Gamma, \tilde\gamma)$ is obtained. If $\partial\Gamma = \emptyset$ then $\tilde\Gamma$ and $\Gamma$ coincide.

Now, given any cyclic permutation $\varepsilon = (\varepsilon _0,\varepsilon _1,\dots,\varepsilon_d = d)$ of the color set $\Delta_d$, a regular imbedding of $\tilde \Gamma$ into a surface $F$ is simply an imbedding $i:|\tilde \Gamma|\to F$ such that the vertices of $\tilde \Gamma$ intersecting $\partial F$ are same as that of new added vertices and the regions of the imbedding are  bounded either by a cycle (internal region)  or by a walk (boundary region) of $\tilde \Gamma$ with $\varepsilon_i,\varepsilon_{i+1}(i$ mod $d+1)$ colored edges alternatively. 

Using Gross `voltage theory' (see \cite{gro74}), in the bipartite case, and Stahl `embedding schemes' (see \cite{st78}), in the non bipartite case, one can prove that given any cyclic permutation $\varepsilon$ of $\Delta_d$ there exists a regular embedding $i_\varepsilon : \tilde \Gamma \hookrightarrow F_\varepsilon$ , where the surface $F_\varepsilon$ has euler characteristic 
\begin{eqnarray}\label{relation_chi}
\chi_{\varepsilon}(\Gamma)= \sum_{i \in \mathbb{Z}_{d+1}} \dot{g}_{\varepsilon_i \varepsilon_{i+1}} + (1-d) \dot{p} + (2-d) \overline{p}
\end{eqnarray}
and $\lambda_\varepsilon=\partial g_{\varepsilon_0\varepsilon_{d-1}}$ holes where $\dot{g}_{ij}$ and $\partial g_{ij}$ denote the number of $\{i,j\}$-colored cycles in $\Gamma$ and $\partial\Gamma$ respectively. For more details we refer \cite{ga87}.

Moreover, if $F_\varepsilon$ is orientable (resp., non-orientable) depending upon $\Gamma$ is bipartite(resp., non bipartite) then the integer
$$\rho_{\varepsilon}(\Gamma) = 1 - \chi_{\varepsilon}(\Gamma)/2-\lambda_\varepsilon/2$$
is same as the genus (resp,. half of the genus) of the surface $F_{\varepsilon}$.

\begin{definition}[\cite{ga87,ga81}]\label{def:regular-genus}
Let $(\Gamma,\gamma) \in \mathbb{G}_d$ be a $(d+1)$-colored graph. Then the non-negative integer 
$$\rho(\Gamma)= \min \{\rho_{\varepsilon}(\Gamma) \ | \  \varepsilon =(\varepsilon_0,\varepsilon_1,\dots,\varepsilon_d = d)\ \text{ is a cyclic permutation of } \ \Delta_d\}$$ is called the  regular genus of $(\Gamma,\gamma)$. For a connected compact PL $d$-manifold $M$ (possibly with boundary), the non-negative integer 
$$\mathcal G(M) = \min \{\rho(\Gamma) \ | \  (\Gamma,\gamma) \mbox{ is a crystallization of } M\}$$
is called the regular genus of the manifold $M$.
\end{definition}

Let $(\Gamma,\gamma)$ be a crystallization of a closed connected $3$-manifold $M$. Then, it is easy to calculate from Proposition \ref{prop:3mfd} that $\rho(\Gamma)=\min\{g_{01}-1,g_{02}-1,g_{12}-1\}$. Let $M$ be a connected compact PL $4$-manifold with $h$ boundary components say $\partial^1M,\dots, \partial^hM$. Then $\mathcal G(\partial M)$ is defined as $\sum_{i=1}^h \mathcal G(\partial^i M)$. Therefore, for any crystallization $(\Gamma,\gamma)$ of $M$, we have 

$$\mathcal G(\partial M)= \sum_{k=1}^h \mathcal G(\partial^k M)\leq \sum_{k=1}^h (\partial^k g_{ij}-1)=\partial g_{ij}-h 
\mbox{ for distinct i,j } \in \Delta_3.$$

\section{Lower bounds for  gem-complexity of PL 4-manifolds with boundary}
Let $(\Gamma,\gamma)$ be a crystallization of a connected compact PL $4$-manifold $M$ with $h$ boundary components. Throughout this article, we assume that $(\Gamma,\gamma)\in \mathbb{G}_4$  and  $h\geq 1$.
For a color set $\{i_1,i_2,\dots,i_k\} \subset \Delta_4$, $g_{i_1i_2 \cdots i_k}$ (or sometimes $g(\Gamma_{\{i_1,i_2, \dots, i_k\}})$) denotes the number of connected components and $\dot g_{i_1i_2 \dots i_k}$ denotes the number of $k$-regular connected components of the graph $\Gamma_{\{i_1, i_2, \dots, i_k\}}$.

\begin{lemma}\label{lemma:interior}
Let $(\Gamma,\gamma)\in \mathbb{G}_4$ be a crystallization of a connected compact PL $4$-manifold $M$ with $h$ boundary components. Then for any $a,b\in \{0,1,2,3\}$, $\dot g_{ab4}\geq h-1$.
\end{lemma}

\begin{proof}
For $c\in \{0,1,2,3\}$, $\Gamma_{\Delta_4\setminus \{c\}}$ has $h$ components and each of the components of  $\Gamma_{\Delta_4\setminus \{c\}}$ represents a 3-ball $\mathbb{D}^3$. Let 
 $\Gamma^1,\dots,\Gamma^h$ be the components of $\Gamma_{\Delta_4\setminus \{c\}}$  such that $\Gamma^j$ is connected with $\Gamma^i$ by an edge of color $c$, for $2\leq j \leq h$, and for some $i<j$. Choose $h-1$ number of such pairs $(\Gamma^i,\Gamma^{j})$, for $2\leq j \leq h$ and for  a fixed $i<j$. Let $\partial^q \Gamma$ be the component of $\partial \Gamma$ such that $V(\partial^q \Gamma)=V(\partial \Gamma^q)$, for $1\leq q \leq h$. Note that $\mathcal K(\Gamma)$ has exactly $h$ number of vertices labeled by the color $c$, and each boundary component contains exactly one vertex labeled by the color $c$,  for $c\in\{0,1,2,3\}$. We wish to show that $\mathcal K(\Gamma)$ has at least $h-1$ number of edges labeled by the colors $c, l$ for $l \neq c \in \{0,1,2,3\}$. In other words, we prove the existence of $h-1$ number of  interior vertices labeled by the color $l$ in $\mathcal K (\Gamma_{\Delta_4\setminus \{c\}})$ for $l \neq c \in \{0,1,2,3\}$.
    
Choose such a pair $(\Gamma^i,\Gamma^{j})$.  Let $v_iv_j$ be an edge of color $c$ in $\Gamma$ such that $v_i\in V(\Gamma^i)$ and $v_j\in V(\Gamma^j)$. Therefore, the 2-simplexes of  $\sigma(v_i)$ and $\sigma(v_j)$ labeled by the color set $\Delta_4\setminus\{c,4\}$ are identified in $\mathcal{K}(\Gamma)$.  For  $l\in \Delta_4\setminus\{c,4\}$, let $x_l$ be the $l$-colored vertex of $\sigma(v_i)$, which is same as the $l$-colored vertex of $\sigma(v_j)$. Let $x_l^i$ and $x_l^j$ be copies of $x_l$, which are the $l$-colored vertices in $\mathcal{K}(\Gamma^i)$ and $\mathcal{K}(\Gamma^j)$ respectively. Since the boundary components are disjoint, $x_l$ is identified with exactly one $l$-colored boundary vertex in $\mathcal{K}(\Gamma)$. If  $x_l$ is identified with the $l$-colored vertex of $\mathcal{K}(\partial^i\Gamma)$ (resp., $\mathcal{K}(\partial^j\Gamma)$) then $x_l^j$ (resp., $x_l^i$) is an interior vertex (as it can not be a boundary vertex) in $\mathcal{K}(\Gamma^j)$ (resp., $\mathcal{K}(\Gamma^i)$). If  $x_l$ is identified with the $l$-colored vertex of $\mathcal{K}(\partial^q\Gamma)$ for some $q\neq i,j$ then $x_l^i$ and $x_l^j$ both are interior  vertices in $\mathcal{K}(\Gamma^i)$ and $\mathcal{K}(\Gamma^j)$ respectively. 

If possible let there exists $k> j$ such that $\Gamma^{k}$ is connected with $\Gamma^{j}$ by some edges of color $c$. let $u_ju_k$ be an edge of color $c$ in $\Gamma$ such that $u_j\in V(\Gamma^j)$ and $u_k\in V(\Gamma^k)$. For  $l\in \Delta_4\setminus\{c,4\}$, let $y_l$ be the identified $l$-colored vertex of $\sigma(u_j)$ and $\sigma(u_k)$. Let $y_l^j$ and $y_l^k$ be copies of $y_l$, which are the $l$-colored vertices in $\mathcal{K}(\Gamma^j)$ and $\mathcal{K}(\Gamma^k)$ respectively. Let $x_l,x_l^i,x_l^j$ be as above for the pair $(\Gamma^i,\Gamma^{j})$.  If $x_l^j = y_l^j$ then the corresponding  vertices $x_l$ and $y_l$ are identified in $\mathcal{K}(\Gamma)$, and they are identified with exactly one $l$-colored boundary vertex in $\mathcal{K}(\Gamma)$. Thus, if $x_l \neq y_l$ are different then $x_l^j \neq y_l^j$ in $\mathcal{K}(\Gamma^j)$, and hence the $l$-colored interior vertices in $\mathcal K (\Gamma_{\Delta_4\setminus \{c\}})$ given by the pairs $(\Gamma^i,\Gamma^{j})$ and $(\Gamma^j,\Gamma^{k})$ are also different. Similar arguments hold if  $\Gamma^{k}$ is connected with $\Gamma^{i}$ by some edges of color $c$, i.e., for the pairs  $(\Gamma^i,\Gamma^{j})$ and $(\Gamma^i,\Gamma^{k})$.

For $1\leq q \leq h$, let $n_q$ denote the number of pairs  $(\Gamma^i,\Gamma^{j})$ such that the $l$-colored vertex $x_l$ is identified with the $l$-colored vertex of $\mathcal{K}(\partial^q\Gamma)$.
Then $n_q\geq 0$ and $n_1+n_2+\cdots+n_h=h-1$. From the above arguments it is clear that, the $n_q$ number of pairs  $(\Gamma^i,\Gamma^{j})$ give at least $n_q$ number of $l$-colored interior vertices in $\mathcal K (\Gamma_{\Delta_4\setminus \{c\}})$. Therefore, $\mathcal K (\Gamma_{\Delta_4\setminus \{c\}})$ has at least $h-1$ number of interior vertices labeled by the color $l$, where $l\neq c \in\{0,1,2,3\}$. This implies, $\dot g_{ab4}\geq h-1$ for  $a,b\in \{0,1,2,3\}$.
\end{proof}

\begin{lemma}\label{lemma:first}
Let $(\Gamma,\gamma)\in \mathbb{G}_4$ be a 5-colored graph which represents a connected compact PL $4$-manifold with boundary. Then
\begin{enumerate}[$(i)$]
\item $g_{ij4}=\dot{g}_{ij4}+\partial{g_{ij}}$,  
\item $g_{i4}=\dot{g}_{i4}+\overline{p}$.
\end{enumerate}  
\end{lemma}

\begin{proof}
Let $(\Gamma,\gamma)$ represent a connected compact PL $4$-manifold $M$ with boundary $\partial M$. Let $(\partial \Gamma, \partial \gamma)$ be the boundary graph of $(\Gamma,\gamma)$. Then 
 $(\partial \Gamma, \partial \gamma)$  represents $\partial M$. Let $\mathcal{K}(\Gamma)$ and $\mathcal{K}(\partial \Gamma)$  be the simplicial cell complexes corresponding to $(\Gamma,\gamma)$ and $(\partial \Gamma, \partial \gamma)$ respectively. Then $\mathcal{K}(\partial \Gamma)$  is a sub-complex of  $\mathcal{K}(\Gamma)$.  Let $\{k,l\}=\Delta_4 \setminus\{i,j,4\}$. Then $g_{ij4}$ 
denotes the number of edges between the vertices labeled by the colors $k$ and $l$ in $\mathcal{K}(\Gamma)$ and $\dot g_{ij4}$  denotes the number of interior edges between the vertices labeled by the colors $k$ and $l$ in $\mathcal{K}(\Gamma)$. Further,  $\partial g_{ij}$  denotes the number of edges between the vertices labeled by the colors $k$ and $l$ in $\mathcal{K}(\partial \Gamma)$, i.e.,  the number of boundary edges between the vertices labeled by the colors $k$ and $l$ in $\mathcal{K}(\Gamma)$. Therefore, $g_{ij4}=\dot{g}_{ij4}+\partial{g_{ij}}$. This proves Part $(i)$.

Let $\{j,k,l\}=\Delta_4 \setminus\{i,4\}$. Then $g_{i4}$ denotes the number of triangles with vertices labeled by the colors $j, k$ and $l$ in $\mathcal{K}(\Gamma)$ and $\dot g_{i4}$  denotes the number of interior triangles with vertices labeled by the colors $j, k$ and $l$ in $\mathcal{K}(\Gamma)$. Further, $\mathcal{K}(\partial \Gamma)$ has exactly $\bar p$ number of triangles with vertices labeled by any three fixed colors. This implies, the number of boundary triangles with vertices labeled by the colors $j, k$ and $l$ in $\mathcal{K}(\Gamma)$ is $\bar p$.
 Therefore, $g_{i4}=\dot{g}_{i4}+\overline{p}$. This proves Part $(ii)$.
\end{proof}

Let $(\Gamma,\gamma)\in \mathbb{G}_4$ be a 5-colored graph which represents a connected compact PL $4$-manifold $M$ with boundary $\partial M$. Let $2M$ be the closed connected PL $4$-manifold obtained by gluing two copies of $M$ alongwith their boundaries $\partial M$. The closed connected manifold is called the double of the manifold $M$. It is easy to construct a 5-colored graph $(\Gamma',\gamma')$ which represents the closed connected PL $4$-manifold $2M$. Take two distinct copies  $\Gamma^{1}$ and $\Gamma^{2}$ of $\Gamma$ with vertices labeled by $v^{1}_{i}$ and $v^{2}_{i}$ respectively, where $v^{1}_{i}$ and $v^{2}_{i}$ are copies of the vertex $v$ of $\Gamma$. Now, if $v_j$ is a boundary vertex of $\Gamma$ then add a $4$-colored edge between $v^{1}_{j}$ and $v^{2}_{j}$. Since the boundary vertex of  $\Gamma$ represents a boundary facet in $\mathcal{K}(\Gamma)$, by joining two copies of boundary vertex by a $4$-colored edge the corresponding two copies of boundary facets get identified. Thus, the geometric carrier of $\mathcal{K}(\Gamma')$ is $2M$. For $\{i_0,i_1, \dots, i_k\} \subset \Delta_4$, let $g'_{i_0i_1 \cdots i_k}$ denote the number of connected components of $\Gamma^{'}_{\{i_0,i_1,\dots,i_k\}}$. We shall say that the 5-colored graph $(\Gamma',\gamma')$ which represents $2M$ is induced from $(\Gamma,\gamma)$. 

\begin{lemma}\label{lemma:relgprimeg}
Let $(\Gamma,\gamma)\in \mathbb{G}_4$ be a 5-colored graph which represents a connected compact PL $4$-manifold $M$ with boundary and let $(\Gamma^{\prime},\gamma')$ be induced from $\Gamma$ which represents $2M$.   Then for $i,j,k \in \{0,1,2,3\} $,
\begin{enumerate} [$(i)$]
\item $ g^{\prime}_{ijk}= 2 g_{ijk}$ and $g^{\prime}_{ij4}= g_{ij4}+\dot{g}_{ij4}$, 
\item  $g^{\prime}_{ij}= 2 g_{ij}$ and $g^{\prime}_{i4}= g_{i4}+\dot{g}_{i4}.$
 \end{enumerate} 
\end{lemma}

\begin{proof}
For $i,j,k \in \{0,1,2,3\}$, each component of $\Gamma_{i,j,k}$ is regular. Since $\Gamma^{\prime}$ consists of two copies of $\Gamma$ which are joined by $4$-colored edges, we have $ g^{\prime}_{ijk}= 2 g_{ijk}$. 
On the other hand, from construction of $\Gamma^{\prime}$ it is clear that
\begin{align*}
 g^{\prime}_{ij4}=&2 \times \mbox{(number of the regular components of } \Gamma_{ij4}) \\
 &+ \mbox{(number of the non-regular components of } \Gamma_{ij4})\\
 = &2 \dot{g}_{ij4} +(g_{ij4}- \dot{g}_{ij4})\\
 =&g_{ij4}+ \dot{g}_{ij4}\\
\end{align*}
This proves part $(i)$. Part $(ii)$ can be proved by similar arguments as above.
\end{proof}

\begin{lemma}
Let $M$ be a connected compact $4$-manifold with boundary. Let $2M$ be the double of the manifold $M$. Then  $\chi(2M)=2\chi(M)$
\end{lemma}

\begin{proof}
We know that, $\chi(2M)=\chi(M)+\chi(M)-\chi(\partial M)$. Since $\partial M$ is a closed 3-manifold, $\chi(\partial M)=0$. Therefore, $\chi(2M)=2\chi(M)$.  
\end{proof}

\begin{lemma}\label{lemma:rel2pgamma}
Let $(\Gamma,\gamma)\in \mathbb{G}_4$ be a crystallization of a connected compact PL $4$-manifold $M$ with $h$ boundary components. Let $(\Gamma^{\prime},\gamma')$ be the $5$-colored graph which represents $2M$ induced from $(\Gamma,\gamma)$.  Then, $|V(\Gamma)|=6\chi(M)+\sum_{0\leq i<j<k \leq 4}g^{\prime}_{ijk}-12h-6$.
\end{lemma}

\begin{proof}
Let $2p$ be the number of vertices of $\Gamma$. Then, $f_0(\mathcal{K}(\Gamma))=4h+1$ and $f_4(\mathcal{K}(\Gamma))=2p$. Let $X=\mathcal{K}(\Gamma^{\prime})$ be the simplicial cell complex corresponding to $(\Gamma^{\prime},\gamma')$. Therefore, we have $f_0(X)=4h+2$ and $f_4(X)=|V(\Gamma^{\prime})|=4p$. Then, the Dehn Sommerville equations in four dimension give 
\begin{eqnarray*}
f_0 (X)-f_1 (X)+f_2 (X)-f_3 (X)+f_4 (X)=\chi(2M),\\
2f_1(X)-3f_2(X)+4f_3(X)-5f_4(X)=0,\\
2f_3(X)-5f_4(X)=0.\\
\end{eqnarray*} 
These  give 
\begin{eqnarray*}
2f_1(X)-3f_2(X)+5f_4(X)=0,\\
2f_0(X)-2f_1(X)+2f_2(X)-3f_4(X)=2\chi(2M).
\end{eqnarray*}
This implies,
 \begin{equation*}
 6f_0(X)-2f_1(X)+f_4(X)=6\chi(2M).
 \end{equation*}
 Substituting values of $f_0(X)$, $f_4(X)$ and $\chi(2M)=2\chi(M)$, we get $$2p=6\chi(M)+f_1(X)-12h-6.$$
 Thus, $|V(\Gamma)|=6\chi(M)+\sum_{0\leq i<j<k \leq 4}g^{\prime}_{ijk}-12h-6$.
\end{proof}

\begin{lemma}\label{lemma:2pand2pbar}
Let $(\Gamma,\gamma)\in \mathbb{G}_4$ be a crystallization of a connected compact PL $4$-manifold $M$ with $h$ boundary components. Let $2p$ and $2\overline{p}$ be the number of vertices of $\Gamma$ and $\partial \Gamma$ respectively. Then, $2p+2\bar p=6\chi(M)+2\sum_{0 \leq i<j<k \leq 4} g_{ijk}-16h-6$.
\end{lemma}

\begin{proof}
Since $\Gamma$ is a crystallization of $M$, each component of $\partial\Gamma$ is a crystallization of one of the components of $\partial M$. For $1\leq q \leq h$, let $\partial^{q} \Gamma$ represent $q^{th}$ component of $\partial \Gamma$ and $\partial^{q} g_{ij}$ denote the number of cycles with colors $i,j$ in $\partial^{q} \Gamma$. By the characterization of a crystallization of a closed connected $3$-manifold (cf. Proposition \ref{prop:3mfd}), we get
\begin{equation*}
\partial^{q}g_{ij} +\partial^{q}g_{ik}+\partial^{q}g_{jk}=2+|V(\partial^{q} \Gamma)|/2, \mbox{ for } i,j,k \in \Delta_{3}. 
\end{equation*}
As this is true for each $q$, taking summation over $1\leq q \leq h$, we have
\begin{eqnarray}
\partial g_{ij} +\partial g_{ik}+\partial g_{jk}=2h+|V(\partial \Gamma)|/2, \mbox{ for } i,j,k \in \Delta_{3}.\label{eqncomponent}
\end{eqnarray}
By taking summation of all the four equations obtained in Eq. \eqref{eqncomponent}, we have
\begin{eqnarray}
\sum_{0 \leq i<j \leq 3} \partial g_{ij}= 4h+2\overline{p}.\label{eqn:adding}
\end{eqnarray}
Now, from Lemma \ref{lemma:rel2pgamma}, we have

\begin{align*}
2p &=6\chi(M)+\sum_{0\leq i<j<k \leq 4} g^{\prime}_{ijk}-12h-6 \\
&=6\chi(M)+2\sum_{0\leq i<j<k \leq 3} g_{ijk} + \sum_{0 \leq i<j \leq 3} (g_{ij4}+\dot g_{ij4}) -12h-6 \mbox{ (by applying   Lemma \ref{lemma:relgprimeg})}\\ 
&=6\chi(M)+2\sum_{0\leq i<j<k \leq 4} g_{ijk}-\sum_{0 \leq i<j \leq 3} \partial g_{ij} -12h-6 \mbox{ (by applying   Lemma \ref{lemma:first})}\\
&=6\chi(M)+2\sum_{0\leq i<j<k \leq 4}g_{ijk}-2\overline{p}-16h-6  \mbox{ (from Eq. \eqref{eqn:adding})}.
\end{align*}
And the result follows.
\end{proof}

\begin{lemma}\label{lemma:bound}
Let $(\Gamma,\gamma)\in \mathbb{G}_4$ be a crystallization of a connected compact PL $4$-manifold $M$ with $h$ boundary components. Let $2p$ and $2\overline{p}$ be the number of vertices of $\Gamma$ and $\partial \Gamma$ respectively. Let $m$ be the rank of of the fundamental group of the manifold $M$. Then,
\begin{enumerate}[$(i)$]
\item $2p \geq 6\chi(M)+14m+14h-18$,
\item $ 2p+ 2\bar p \geq 6\chi(M)+20m+16h-18$,
\item $ 2p- 2\bar p\geq 6\chi(M)+8m+12h-18$.
\end{enumerate}
\end{lemma}
\begin{proof}
Because $\Gamma$ is a crystallization of a connected compact PL $4$-manifold $M$ with $h$ boundary components, we have $g(\Gamma_{\{0,1,2,3\}})=1$ and $g(\Gamma_{\Delta_4\setminus\{c\}})=h$, for $0 \leq c \leq 3$. From Proposition \ref{prop:rank}, we know that for $a,b \in \Delta_4$, 
$$m \leq g(\Gamma_{\Delta_4\setminus\{a,b\}})-\big(g(\Gamma_{\Delta_4\setminus\{a\}})+g(\Gamma_{\Delta_4\setminus\{b\}})-1\big).$$ 
This implies,
$$g_{ijk}\geq m+h \mbox{ and } g_{ij4}\geq m+2h-1 \mbox{ for } i,j,k \in \{0,1,2,3\}.$$
Therefore, from Lemma \ref{lemma:rel2pgamma}, we have 

\begin{align*}
2p &=6\chi(M)+\sum_{0\leq i<j<k \leq 4} g^{\prime}_{ijk}-12h-6, \\
&=6\chi(M)+2\sum_{0\leq i<j<k \leq 3} g_{ijk} + \sum_{0 \leq i<j \leq 3} (g_{ij4}+\dot g_{ij4}) -12h-6 \mbox{ (by applying   Lemma \ref{lemma:relgprimeg})}\\ 
&\geq 6\chi(M)+8(m+h)+6(m+2h-1)+6(h-1)-12h-6\\
&=6\chi(M)+14m+14h-18.
\end{align*}
This proves Part $(i)$.

Further, from Lemma \ref{lemma:2pand2pbar} we have
\begin{align*}
2p+2\overline{p}&=6\chi(M)+2\sum_{0 \leq i<j<k \leq 4} g_{ijk}-16h-6\\
&\geq 6\chi(M) +2\{4(m+h)+6(m+2h-1)\}-16h-6\\
& =6\chi(M)+20m+16h-18.
\end{align*}
This proves Part $(ii)$.

From Eq. \eqref{eqn:adding}, we have
\begin{align*}
2\overline{p}&= \sum_{0 \leq i<j \leq 3} \partial g_{ij}-4h.\\
&= \sum_{0 \leq i<j \leq 3} (g_{ij4} - \dot g_{ij4}) -4h \mbox{ (by applying  Lemma \ref{lemma:first}) }
\end{align*}
Therefore,
\begin{align*}
2p -2\bar p = & 6\chi(M)+\sum_{0\leq i<j<k \leq 4} g^{\prime}_{ijk}-12h-6 -  \sum_{0 \leq i<j \leq 3} g_{ij4}+  \sum_{0 \leq i<j \leq 3} \dot g_{ij4}+4h \\
=& 6\chi(M)+2\sum_{0\leq i<j<k \leq 3} g_{ijk} + \sum_{0 \leq i<j \leq 3} (g_{ij4}+\dot g_{ij4}) -12h-6 \\ 
& -  \sum_{0 \leq i<j \leq 3} g_{ij4}+  \sum_{0 \leq i<j \leq 3} \dot g_{ij4}+4h \\
= &6\chi(M)+2\sum_{0\leq i<j<k \leq 3} g_{ijk}+  2\sum_{0 \leq i<j \leq 3} \dot g_{ij4} -8h-6  \\ 
\geq &6\chi(M)+8(m+h)+12(h-1)-8h-6 \\
=&6\chi(M)+8m+12h-18.
\end{align*}
This proves Part $(iii)$.
\end{proof}

Let $M$ be a connected compact PL $4$-manifold with $h$ boundary components. Then the gem-complexity  $\mathit{k}(M) $ of  $M$ is defined as the non-negative integer $p - 1$, where $2p$ is the minimum number of vertices of a crystallization of $M$. Thus, we have the following result.

\begin{theorem}\label{theorem:main1}
Let $M$ be a connected compact PL $4$-manifold with $h$ boundary components. Let $m$ be the rank of the fundamental group of the manifolds $M$. Then,
\begin{enumerate}[$(i)$] 
\item $\mathit{k}(M)\geq 3\chi(M)+7m+7h-10$,
\item $\mathit{k}(M)\geq \mathit{k}(\partial M)+3\chi(M)+4m+6h-9$.
\end{enumerate} 
\end{theorem}

\begin{figure}[H]
\tikzstyle{vert}=[circle, draw, fill=black!100, inner sep=0pt, minimum width=4pt] \tikzstyle{vertex}=[circle,
draw, fill=black!00, inner sep=0pt, minimum width=4pt] \tikzstyle{ver}=[] \tikzstyle{extra}=[circle, draw,
fill=black!50, inner sep=0pt, minimum width=2pt] \tikzstyle{edge} = [draw,thick,-] \centering
\begin{tikzpicture}[scale=0.5]

\begin{scope}[shift={(9,6)}]
\node[ver] (4) at (1,-5){\tiny{$4$}}; 
\node[ver] (3) at (1,-4.5){\tiny{$3$}}; 
\node[ver] (2) at (1,-4){\tiny{$2$}};
\node[ver](1) at (1,-3.5){\tiny{$1$}}; 
\node[ver](0) at (1,-3){\tiny{$0$}}; 

\node[ver] (9) at (3,-5){}; 
\node[ver] (8) at (3,-4.5){}; 
\node[ver](7) at (3,-4){}; 
\node[ver](6) at (3,-3.5){}; 
\node[ver] (5) at (3,-3){};

\path[edge] (0) -- (5);
\draw[line width=2pt, line cap=rectengle, dash pattern=on 1pt off 1]   plot [smooth,tension=1] (1.2,-3) -- (5);
\path[edge] (1) -- (6);
\draw [line width=3pt, line cap=round, dash pattern=on 0pt off 2\pgflinewidth] (1) -- (6);
\path[edge] (2) -- (7);
\path[edge,dashed] (3) -- (8);
\path[edge,dotted] (4) -- (9);
\end{scope}

\begin{scope}
\foreach \x/\y/\z/\w in {0/1/0.5/v_{2},0/3/3.5/v_{3},2/1/0.5/v_{4},2/3/3.5/v_{5},4/1/0.5/v_{6},4/3/3.5/v_{7}, 6/1/0.5/v_{8},6/3/3.5/v_{9}, 8/1/0.5/v_{10},8/3/3.5/v_{1}}
{\node[ver](\w) at (\x,\z){\tiny{$\w$}};
\node[vert] (\w) at (\x,\y){};}

\foreach \x/\y/\z/\w in {v_{2}/v_{3}/-0.5/2, v_{2}/v_{3}/0/2, v_{4}/v_{5}/1.5/2, v_{6}/v_{7}/3.5/2,  v_{8}/v_{9}/5.5/2,  v_{8}/v_{9}/6/2,   v_{10}/v_{1}/7.5/2,  v_{10}/v_{1}/8/2,  v_{10}/v_{1}/8.5/2}
{\draw[edge] plot [smooth,tension=1] coordinates{(\x) (\z,\w) (\y)};
}

\foreach \x/\y in {v_{2}/v_{4}, v_{3}/v_{5}, v_{5}/v_{7}, v_{6}/v_{8},  v_{6}/v_{4},  v_{7}/v_{9}}
{\draw[edge] (\x) -- (\y);
}

\foreach \x/\y/\z/\w in {v_{6}/v_{7}/3.5/2,  v_{8}/v_{9}/5.5/2,  v_{10}/v_{1}/7.5/2}
{\draw[line width=2pt, line cap=rectengle, dash pattern=on 1pt off 1]   plot [smooth,tension=1] coordinates{(\x) (\z,\w) (\y)};
}

\foreach \x/\y in {v_{2}/v_{4},  v_{3}/v_{5}}
{\draw[line width=2pt, line cap=rectengle, dash pattern=on 1pt off 1]  (\x) -- (\y);
}

\foreach \x/\y/\z/\w in {v_{2}/v_{3}/-0.5/2,   v_{8}/v_{9}/6/2,  v_{10}/v_{1}/8/2}
{\draw [line width=3pt, line cap=round, dash pattern=on 0pt off 2\pgflinewidth]  plot [smooth,tension=1] coordinates{(\x) (\z,\w) (\y)};
}

\foreach \x/\y in {v_{6}/v_{4},  v_{7}/v_{5}}
{\draw [line width=3pt, line cap=round, dash pattern=on 0pt off 2\pgflinewidth]  (\x) -- (\y);
}

\foreach \x/\y/\z/\w in {v_{2}/v_{3}/0.5/2, v_{4}/v_{5}/2/2, v_{6}/v_{7}/4/2}
{\draw[edge, dashed] plot [smooth,tension=1] coordinates{(\x) (\z,\w) (\y)};
}

\foreach \x/\y/\z/\w in {v_{6}/v_{7}/4.5/2, v_{4}/v_{5}/2.5/2, v_{8}/v_{9}/6.5/2}
{\draw[edge, dotted] plot [smooth,tension=1] coordinates{(\x) (\z,\w) (\y)};
}

\foreach \x/\y in {v_{1}/v_{9}, v_{8}/v_{10}}
{\draw[edge, dashed] (\x) -- (\y);
}
\draw[edge, dotted] plot [smooth,tension=1] coordinates{(v_{1}) (9,2) (v_{10})};
\draw[edge, dotted] plot [smooth,tension=1] coordinates{(v_{2}) (-1,2) (v_{3})};
\end{scope}

\end{tikzpicture}
\caption{A  crystallization $(\Gamma,\gamma)$ of $\mathbb{S}^4$.}\label{fig:1}
\end{figure}
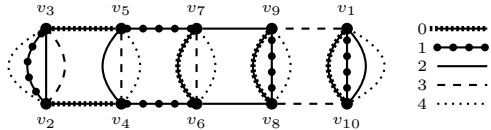

\begin{theorem}\label{theorem:connected-sum-1}
For $1\leq i \leq 2$, let $M_i$ be a connected compact PL $4$-manifold with $h_i$ boundary components. If the bound obtained in Part $(i)$ (resp., Part $(ii)$) of Theorem $\ref{theorem:main1}$ is sharp for both $M_1$ and $M_2$ then the bound is sharp for $M_1\# M_2$ as well.
\end{theorem}

\begin{proof}
For $1\leq i \leq 2$, let $m_i$ be the rank of the fundamental groups of the manifolds $M_i$. Then, for $1\leq i \leq 2$,
$$\mathit{k}(M_i) = 3\chi(M_i)+7m_i+7h_i-10.$$
Let $(\Gamma^i,\gamma^i)\in \mathbb{G}_4$ be a crystallization of $M_i$ with $6\chi(M_i)+14m_i+14h_i-18$ vertices. Choose  interior vertices $u\in V(\Gamma^1) $ and $v\in V(\Gamma^2)$.
Let $(\Gamma,\gamma)$ be the crystallization of $\mathbb{S}^4$ constructed as in Figure \ref{fig:1}. Let  $v_1,v_2 \in V(\Gamma)$ be as in Figure \ref{fig:1}. Then the graph connected sum $\Gamma^1\#_{uv_1} \Gamma \#_{vv_2} \Gamma^2$  gives a crystallization of $M_1\#M_2$ with $(6\chi(M_1)+14m_1+14h_1-18)+(6\chi(M_2)+14m_2+14h_2-18)+6=6\chi(M_1\# M_2)+ 14(m_1+m_2)+14(h_1+h_2)-18$ vertices. Therefore, $\mathit{k}(M_1\#M_2) \leq 3\chi(M_1\#M_2)+7(m_1+m_2)+7(h_1+h_2)-10$. On the other hand, by Part (i) of Theorem \ref{theorem:main1}, we have 
$\mathit{k}(M_1\#M_2) \geq 3\chi(M_1\#M_2)+7(m_1+m_2)+7(h_1+h_2)-10$. Therefore, $\mathit{k}(M_1\#M_2) = 3\chi(M_1\#M_2)+7(m_1+m_2)+7(h_1+h_2)-10$.

Now, assume that $\mathit{k}(M_i) = \mathit{k}(\partial M_i)+3\chi(M_i)+4m_i+6h_i-9$, for $1\leq i \leq 2$. Let $(\Gamma^i,\gamma^i)\in \mathbb{G}_4$ be a crystallization of $M_i$ with $6\chi(M_i)+2\mathit{k}(\partial M_i)+8m_i+12h_i-16$ vertices. Since $\partial M_1$ and $\partial M_2$ are disjoint, $\mathit{k}(\partial (M_1\#M_2))= \mathit{k}(\partial M_1 \sqcup \partial M_2))= \mathit{k}(\partial M_1)+\mathit{k}(\partial M_2)+1$. Thus, by the similar approach as above, we get $\mathit{k}(M_1\# M_2) = \mathit{k}(\partial (M_1\#M_2))+3\chi(M_1\#M_2)+4(m_1+m_2)+6(h_1+h_2)-9$.
\end{proof}

\begin{remark}\label{remark:gem-complexity}
{\rm
Let $M$ be a closed connected PL 4-manifold which admits semi-simple crystallization constructed in \cite{bc15}. Then, by removing few appropriate edges of color $4$, it is easy to construct connected compact PL 4-manifolds with spherical or non-spherical boundary, for which the bounds of Theorem \ref{theorem:main1} are sharp. In particular, by removing few edges of color $4$ from the unique 10-vertex crystallization of $\mathbb{S}^3\times \mathbb{S}^1$, we get 10-vertex crystallizations of  $\mathbb{S}^3\times [0,1]$ and $\mathbb{D}^3\times \mathbb{S}^1$ (cf. Figures \ref{fig:2} and \ref{fig:3}). Further, by removing few edges of color $4$ from the standard $16$-vertex crystallization of $\mathbb{RP}^4$, we get a  connected compact PL 4-manifolds whose boundary is $\TPSS$ (cf. Figure \ref{fig:4}). 
Moreover, by removing few edges of color $4$ from the simple crystallizations of $\mathbb{CP}^2$, $\mathbb{S}^2\times \mathbb{S}^2$ and $\mathbb{K}_3$ surface (cf. \cite{bs16}), we get  crystallizations of $\mathbb{CP}^2\# \mathbb{D}^4$, $(\mathbb{S}^2\times \mathbb{S}^2)\# \mathbb{D}^4$ and $\mathbb{K}_3\# \mathbb{D}^4$ respectively. In each of the above cases, the minimality of gem-complexity follows from Theorem \ref{theorem:main1}.

Let $M_1$ and $M_2$ be any two connected compact PL $4$-manifolds with boundary as constructed above. It follows from Theorem  \ref{theorem:connected-sum-1} that the bounds of Theorem $\ref{theorem:main1}$ are sharp for $M_1\# M_2$ as well. Thus our bounds for gem-complexity is sharp for a huge class of connected compact PL $4$-manifolds with boundary - which comprehends all  connected compact PL 4-manifolds with boundary as above, together with their connected sums, possibly by taking copies with reversed orientation, too.
}
\end{remark}

\begin{figure}[H]
\tikzstyle{vert}=[circle, draw, fill=black!100, inner sep=0pt, minimum width=4pt] \tikzstyle{vertex}=[circle,
draw, fill=black!00, inner sep=0pt, minimum width=4pt] \tikzstyle{ver}=[] \tikzstyle{extra}=[circle, draw,
fill=black!50, inner sep=0pt, minimum width=2pt] \tikzstyle{edge} = [draw,thick,-] \centering
\begin{tikzpicture}[scale=0.5]

\begin{scope}[shift={(9,6)}]
\node[ver] (4) at (1,-5){\tiny{$4$}}; 
\node[ver] (3) at (1,-4.5){\tiny{$3$}}; 
\node[ver] (2) at (1,-4){\tiny{$2$}};
\node[ver](1) at (1,-3.5){\tiny{$1$}}; 
\node[ver](0) at (1,-3){\tiny{$0$}}; 

\node[ver] (9) at (3,-5){}; 
\node[ver] (8) at (3,-4.5){}; 
\node[ver](7) at (3,-4){}; 
\node[ver](6) at (3,-3.5){}; 
\node[ver] (5) at (3,-3){};

\path[edge] (0) -- (5);
\draw[line width=2pt, line cap=rectengle, dash pattern=on 1pt off 1]   plot [smooth,tension=1] (1.2,-3) -- (5);
\path[edge] (1) -- (6);
\draw [line width=3pt, line cap=round, dash pattern=on 0pt off 2\pgflinewidth] (1) -- (6);
\path[edge] (2) -- (7);
\path[edge,dashed] (3) -- (8);
\path[edge,dotted] (4) -- (9);
\end{scope}

\begin{scope}
\foreach \x/\y/\z/\w in {0/1/0.5/v_{2},0/3/3.5/v_{3},2/1/0.5/v_{4},2/3/3.5/v_{5},4/1/0.5/v_{6},4/3/3.5/v_{7}, 6/1/0.5/v_{8},6/3/3.5/v_{9}, 8/1/0.5/v_{10},8/3/3.5/v_{1}}
{\node[ver](\w) at (\x,\z){\tiny{$\w$}};
\node[vert] (\w) at (\x,\y){};}

\foreach \x/\y/\z/\w in {v_{2}/v_{3}/-0.5/2, v_{2}/v_{3}/0/2, v_{4}/v_{5}/1.5/2, v_{6}/v_{7}/3.5/2,  v_{8}/v_{9}/5.5/2,  v_{8}/v_{9}/6/2,   v_{10}/v_{1}/7.5/2,  v_{10}/v_{1}/8/2,  v_{10}/v_{1}/8.5/2}
{\draw[edge] plot [smooth,tension=1] coordinates{(\x) (\z,\w) (\y)};
}

\foreach \x/\y in {v_{2}/v_{4}, v_{3}/v_{5}, v_{5}/v_{7}, v_{6}/v_{8},  v_{6}/v_{4},  v_{7}/v_{9}}
{\draw[edge] (\x) -- (\y);
}

\foreach \x/\y/\z/\w in {v_{6}/v_{7}/3.5/2,  v_{8}/v_{9}/5.5/2,  v_{10}/v_{1}/7.5/2}
{\draw[line width=2pt, line cap=rectengle, dash pattern=on 1pt off 1]   plot [smooth,tension=1] coordinates{(\x) (\z,\w) (\y)};
}

\foreach \x/\y in {v_{2}/v_{4},  v_{3}/v_{5}}
{\draw[line width=2pt, line cap=rectengle, dash pattern=on 1pt off 1]  (\x) -- (\y);
}

\foreach \x/\y/\z/\w in {v_{2}/v_{3}/-0.5/2,   v_{8}/v_{9}/6/2,  v_{10}/v_{1}/8/2}
{\draw [line width=3pt, line cap=round, dash pattern=on 0pt off 2\pgflinewidth]  plot [smooth,tension=1] coordinates{(\x) (\z,\w) (\y)};
}

\foreach \x/\y in {v_{6}/v_{4},  v_{7}/v_{5}}
{\draw [line width=3pt, line cap=round, dash pattern=on 0pt off 2\pgflinewidth]  (\x) -- (\y);
}

\foreach \x/\y/\z/\w in {v_{2}/v_{3}/0.5/2, v_{4}/v_{5}/2/2, v_{6}/v_{7}/4/2}
{\draw[edge, dashed] plot [smooth,tension=1] coordinates{(\x) (\z,\w) (\y)};
}

\foreach \x/\y/\z/\w in {v_{6}/v_{7}/4.5/2, v_{4}/v_{5}/2.5/2, v_{8}/v_{9}/6.5/2}
{\draw[edge, dotted] plot [smooth,tension=1] coordinates{(\x) (\z,\w) (\y)};
}

\foreach \x/\y in {v_{1}/v_{9}, v_{8}/v_{10}}
{\draw[edge, dashed] (\x) -- (\y);
}
\end{scope}

\end{tikzpicture}
\caption{A  10-vertex (minimal) crystallization of $\mathbb{S}^3 \times [0,1]$.}\label{fig:2}
\end{figure}
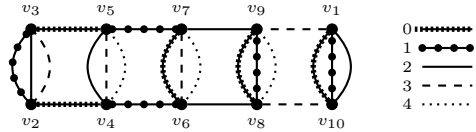

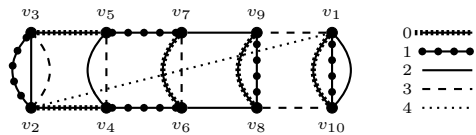
\begin{figure}[H]
\tikzstyle{vert}=[circle, draw, fill=black!100, inner sep=0pt, minimum width=4pt] \tikzstyle{vertex}=[circle,
draw, fill=black!00, inner sep=0pt, minimum width=4pt] \tikzstyle{ver}=[] \tikzstyle{extra}=[circle, draw,
fill=black!50, inner sep=0pt, minimum width=2pt] \tikzstyle{edge} = [draw,thick,-] \centering
\begin{tikzpicture}[scale=0.5]

\begin{scope}[shift={(9,6)}]
\node[ver] (4) at (1,-5){\tiny{$4$}}; 
\node[ver] (3) at (1,-4.5){\tiny{$3$}}; 
\node[ver] (2) at (1,-4){\tiny{$2$}};
\node[ver](1) at (1,-3.5){\tiny{$1$}}; 
\node[ver](0) at (1,-3){\tiny{$0$}}; 

\node[ver] (9) at (3,-5){}; 
\node[ver] (8) at (3,-4.5){}; 
\node[ver](7) at (3,-4){}; 
\node[ver](6) at (3,-3.5){}; 
\node[ver] (5) at (3,-3){};

\path[edge] (0) -- (5);
\draw[line width=2pt, line cap=rectengle, dash pattern=on 1pt off 1]   plot [smooth,tension=1] (1.2,-3) -- (5);
\path[edge] (1) -- (6);
\draw [line width=3pt, line cap=round, dash pattern=on 0pt off 2\pgflinewidth] (1) -- (6);
\path[edge] (2) -- (7);
\path[edge,dashed] (3) -- (8);
\path[edge,dotted] (4) -- (9);
\end{scope}

\begin{scope}
\foreach \x/\y/\z/\w in {0/1/0.5/v_{2},0/3/3.5/v_{3},2/1/0.5/v_{4},2/3/3.5/v_{5},4/1/0.5/v_{6},4/3/3.5/v_{7}, 6/1/0.5/v_{8},6/3/3.5/v_{9}, 8/1/0.5/v_{10},8/3/3.5/v_{1}}
{\node[ver](\w) at (\x,\z){\tiny{$\w$}};
\node[vert] (\w) at (\x,\y){};}

\foreach \x/\y/\z/\w in {v_{2}/v_{3}/-0.5/2, v_{2}/v_{3}/0/2, v_{4}/v_{5}/1.5/2, v_{6}/v_{7}/3.5/2,  v_{8}/v_{9}/5.5/2,  v_{8}/v_{9}/6/2,   v_{10}/v_{1}/7.5/2,  v_{10}/v_{1}/8/2,  v_{10}/v_{1}/8.5/2}
{\draw[edge] plot [smooth,tension=1] coordinates{(\x) (\z,\w) (\y)};
}

\foreach \x/\y in {v_{2}/v_{4}, v_{3}/v_{5}, v_{5}/v_{7}, v_{6}/v_{8},  v_{6}/v_{4},  v_{7}/v_{9}}
{\draw[edge] (\x) -- (\y);
}

\foreach \x/\y/\z/\w in {v_{6}/v_{7}/3.5/2,  v_{8}/v_{9}/5.5/2,  v_{10}/v_{1}/7.5/2}
{\draw[line width=2pt, line cap=rectengle, dash pattern=on 1pt off 1]   plot [smooth,tension=1] coordinates{(\x) (\z,\w) (\y)};
}

\foreach \x/\y in {v_{2}/v_{4},  v_{3}/v_{5}}
{\draw[line width=2pt, line cap=rectengle, dash pattern=on 1pt off 1]  (\x) -- (\y);
}

\foreach \x/\y/\z/\w in {v_{2}/v_{3}/-0.5/2,   v_{8}/v_{9}/6/2,  v_{10}/v_{1}/8/2}
{\draw [line width=3pt, line cap=round, dash pattern=on 0pt off 2\pgflinewidth]  plot [smooth,tension=1] coordinates{(\x) (\z,\w) (\y)};
}

\foreach \x/\y in {v_{6}/v_{4},  v_{7}/v_{5}}
{\draw [line width=3pt, line cap=round, dash pattern=on 0pt off 2\pgflinewidth]  (\x) -- (\y);
}

\foreach \x/\y/\z/\w in {v_{2}/v_{3}/0.5/2, v_{4}/v_{5}/2/2, v_{6}/v_{7}/4/2}
{\draw[edge, dashed] plot [smooth,tension=1] coordinates{(\x) (\z,\w) (\y)};
}


\foreach \x/\y in {v_{1}/v_{9}, v_{8}/v_{10}}
{\draw[edge, dashed] (\x) -- (\y);
}
\draw[edge, dotted] plot [smooth,tension=1] coordinates{(v_{1}) (4,2) (v_{2})};
\end{scope}
\end{tikzpicture}
\caption{A  10-vertex (minimal) crystallization of $\mathbb{D}^3 \times \mathbb{S}^1$.}\label{fig:3}
\end{figure}

\begin{figure}[H]
\tikzstyle{vert}=[circle, draw, fill=black!100, inner sep=0pt, minimum width=4pt] \tikzstyle{vertex}=[circle, draw, fill=black!00, inner sep=0pt, minimum width=4pt] \tikzstyle{ver}=[]
\tikzstyle{extra}=[circle, draw, fill=black!50, inner sep=0pt, minimum width=2pt] \tikzstyle{edge} = [draw,thick,-] \tikzstyle{arrow} = [draw,thick,->]
\centering
\begin{tikzpicture}[scale=0.16]
\begin{scope}
\foreach \x/\y/\z in {-14/4/v1,-10/-4/v10,-6/4/v3,-2/-4/v12,2/4/v5,6/-4/v14,10/4/v7,14/-4/v16}{
 \node[vert] (\z) at (\x,\y){};}

\foreach \x/\y/\z in {-14/-4/v9,-10/4/v2,-6/-4/v11,-2/4/v4,2/-4/v13,6/4/v6,10/-4/v15,14/4/v8}{
 \node[vert] (\z) at (\x,\y){};}

\foreach \x/\y in {v1/v2,v3/v4,v5/v6,v7/v8,v9/v10,v11/v12,v13/v14,v15/v16}{
\path[edge] (\x) -- (\y);
\draw [line width=2pt, line cap=rectengle, dash pattern=on 1pt off 1]  (\x) -- (\y);}

\foreach \x/\y in {v1/v9,v2/v10,v3/v11,v4/v12,v5/v13,v6/v14,v7/v15,v8/v16}{
\path[edge] (\x) -- (\y);
\draw [line width=3pt, line cap=round, dash pattern=on 0pt off 2\pgflinewidth]  (\x) -- (\y);}

\foreach \x/\y in {v2/v3,v6/v7,v10/v11,v14/v15}{
\path[edge] (\x) -- (\y);}
\draw[edge] plot [smooth,tension=1.5] coordinates{(v1)(-8,6)(v4)};
\draw[edge] plot [smooth,tension=1.5] coordinates{(v9)(-8,-6)(v12)};
\draw[edge] plot [smooth,tension=1.5] coordinates{(v5)(8,6)(v8)};
\draw[edge] plot [smooth,tension=1.5] coordinates{(v13)(8,-6)(v16)};

\foreach \x/\y in {v4/v5,v12/v13}{
\path[edge, dotted] (\x) -- (\y);}

\draw[edge, dotted] plot [smooth,tension=1.5] coordinates{(v10)(0,-7)(v15)};
\draw[edge, dotted] plot [smooth,tension=1.5] coordinates{(v9)(0,-8)(v16)};

\foreach \x/\y in {v4/v15,v12/v7,v2/v13,v10/v5}{
\path[edge, dashed] (\x) -- (\y);}
\draw[edge, dashed] plot [smooth,tension=1.5] coordinates{(v1)(-13,-7)(v14)};
\draw[edge, dashed] plot [smooth,tension=1.5] coordinates{(v9)(-13,7)(v6)};
\draw[edge, dashed] plot [smooth,tension=1.5] coordinates{(v8)(13,-7)(v11)};
\draw[edge, dashed] plot [smooth,tension=1.5] coordinates{(v16)(13,7)(v3)};
 \end{scope}

 \begin{scope}[shift={(24,0)}]
\node[ver] (308) at (-3,5){$0$};
\node[ver] (300) at (-3,3){$1$};
\node[ver] (301) at (-3,1){$2$};
\node[ver] (302) at (-3,-1){$3$};
\node[ver] (303) at (-3,-3){$4$};
\node[ver] (309) at (3,5){};
\node[ver] (304) at (3,3){};
\node[ver] (305) at (3,1){};
\node[ver] (306) at (3,-1){};
\node[ver] (307) at (3,-3){};
\path[edge] (300) -- (304);
\path[edge] (308) -- (309);
\draw [line width=2pt, line cap=rectengle, dash pattern=on 1pt off 1]  (308) -- (309);
\draw [line width=3pt, line cap=round, dash pattern=on 0pt off 2\pgflinewidth]  (300) -- (304);
\path[edge] (301) -- (305);
\path[edge, dashed] (302) -- (306);
\path[edge, dotted] (303) -- (307);
\end{scope}

\end{tikzpicture}
\caption{A 16-vertex (minimal) crystallization of a PL 4-manifold with boundary $\TPSS$.}
\label{fig:4}
\end{figure}
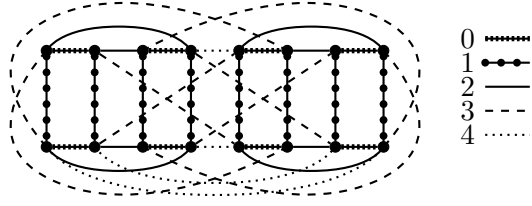

\section{Lower bounds for  regular genus of PL 4-manifolds with boundary}
Recall that, for a connected compact PL $4$-manifold $M$ with boundary, $2M$ denotes the double of the manifold $M$. Let $(\Gamma',\gamma')$ be the 5-colored graph induced from  $(\Gamma,\gamma)$, which represents  $2M$.  For $\{i_0,i_1, \dots, i_k\} \subset \Delta_4$, let $g'_{i_0i_1 \cdots i_k}$ denote the number of connected components of $\Gamma^{'}_{\{i_0,i_1,\dots,i_k\}}$.

\begin{lemma}\label{lemma:rhoepsilon}
Let $(\Gamma,\gamma)\in \mathbb{G}_4$ be a crystallization of a connected compact PL $4$-manifold $M$ with $h$ boundary components. Let $g^{\prime}_{ijk}$ be as above. Then, for any cyclic permutation $\varepsilon=(\varepsilon_0,\dots,\varepsilon_4 =4)$ of the color set $\Delta_4$,\\
$$\rho_\varepsilon (\Gamma)=-1-4h+2\chi(M)+\frac{1}{2}\sum_{i \in \mathbb{Z}_5}g^{\prime}_{\varepsilon_i \varepsilon_{i+2} \varepsilon_{i+4}}-\frac{\partial g_{\varepsilon_0 \varepsilon_3}}{2}.$$
\end{lemma}

\begin{proof}
Let $g^{\prime}_{ij}$ be the number of components of $\Gamma'_{\{i,j\}}$. Let $|V(\Gamma)|=2p$ and $|V(\partial\Gamma)|=2\overline{p}$. Then $|V(\Gamma^{\prime})|=4p$. By the concept of regular genus of crystallization for manifold with boundary, for $\varepsilon=(\varepsilon_0,\dots,\varepsilon_4=4)$,\\
\begin{equation}
\chi_\varepsilon(\Gamma)=\sum_{i \in \mathcal{Z}_5} \dot{g}_{\varepsilon_i \varepsilon_{i+1}} -3(p-\overline{p})-2\overline{p}\label{eqn:chigamma}\\
=\sum_{i \in \mathcal{Z}_5} \dot{g}_{\varepsilon_i \varepsilon_{i+1}}-3p+\overline{p}.
\end{equation} 

Since $(\Gamma',\gamma')$ represents the closed connected PL 4-manifold $2M$, by Proposition \ref{prop:preliminaries} we have
$g^{\prime}_{ij}+g^{\prime}_{ik}+g^{\prime}_{jk}=2g^{\prime}_{ijk}+2p$ for distinct $i,j,k \in \Delta_4$.
This gives ten linear equations which can be written in matrix form as 
$$AX=B,$$
 where 
 {\small
$$A=\begin{bmatrix}
1&1&0&0&1&0&0&0&0&0\\
1&0&1&0&0&1&0&0&0&0\\
1&0&0&1&0&0&1&0&0&0\\
0&1&1&0&0&0&0&1&0&0\\
0&1&0&1&0&0&0&0&1&0\\
0&0&1&1&0&0&0&0&0&1\\
0&0&0&0&1&1&0&1&0&0\\
0&0&0&0&1&0&1&0&1&0\\
0&0&0&0&0&1&1&0&0&1\\
0&0&0&0&0&0&0&1&1&1\\
\end{bmatrix}, 
X=\begin{bmatrix}
g^{\prime}_{01}\\
g^{\prime}_{02}\\
g^{\prime}_{03}\\
g^{\prime}_{04}\\
g^{\prime}_{12}\\
g^{\prime}_{13}\\
g^{\prime}_{14}\\
g^{\prime}_{23}\\
g^{\prime}_{24}\\
g^{\prime}_{34}\\
\end{bmatrix}
 \mbox{and } B=\begin{bmatrix}
2p+2g^{\prime}_{012}\\
2p+2g^{\prime}_{013}\\
2p+2g^{\prime}_{014}\\
2p+2g^{\prime}_{023}\\
2p+2g^{\prime}_{024}\\
2p+2g^{\prime}_{034}\\
2p+2g^{\prime}_{123}\\
2p+2g^{\prime}_{124}\\
2p+2g^{\prime}_{134}\\
2p+2g^{\prime}_{234}\\
\end{bmatrix}.$$}
Since $|A|\neq 0$,
$$X=A^{-1}B,$$
where
{\small $$A^{-1}=\begin{bmatrix}
1/3&1/3&1/3&-1/6&-1/6&-1/6&-1/6&-1/6&-1/6&1/3\\
1/3&-1/6&-1/6&1/3&1/3&-1/6&-1/6&-1/6&1/3&-1/6\\
-1/6&1/3&-1/6&1/3&-1/6&1/3&-1/6&1/3&-1/6&-1/6\\
-1/6&-1/6&1/3&-1/6&1/3&1/3&1/3&-1/6&-1/6&-1/6\\
1/3&-1/6&-1/6&-1/6&-1/6&1/3&1/3&1/3&-1/6&-1/6\\
-1/6&1/3&-1/6&-1/6&1/3&-1/6&1/3&-1/6&1/3&-1/6\\
-1/6&-1/6&1/3&1/3&-1/6&-1/6&-1/6&1/3&1/3&-1/6\\
-1/6&-1/6&1/3&1/3&-1/6&-1/6&1/3&-1/6&-1/6&1/3\\
-1/6&1/3&-1/6&-1/6&1/3&-1/6&-1/6&1/3&-1/6&1/3\\
1/3&-1/6&-1/6&-1/6&-1/6&1/3&-1/6&-1/6&1/3&1/3\\
\end{bmatrix}.
$$}
Consider $$B=M+L,$$
where

$$M=\begin{bmatrix}
2p\\
2p\\
2p\\
2p\\
2p\\
2p\\
2p\\
2p\\
2p\\
2p\\
\end{bmatrix}
\mbox{ and } L=
\begin{bmatrix}
2g^{\prime}_{012}\\
2g^{\prime}_{013}\\
2g^{\prime}_{014}\\
2g^{\prime}_{023}\\
2g^{\prime}_{024}\\
2g^{\prime}_{034}\\
2g^{\prime}_{123}\\
2g^{\prime}_{124}\\
2g^{\prime}_{134}\\
2g^{\prime}_{234}\\
\end{bmatrix}.
$$
Thus, $$X=A^{-1}M+A^{-1}L.$$
Therefore, $g^{\prime}_{ij}=\frac{2p}{3}+2\sum_{0 \leq r<s<t \leq 4} c^{ij}_{rst} g^{\prime}_{rst}$, where $c^{ij}_{rst}$ is the member of $A^{-1}$ corresponding to ${ij}$-row and $rst$-column of $A^{-1}$.
 Now, for given permutation  $\varepsilon=(\varepsilon_0,\dots,\varepsilon_4 =4)$ of the color set $\Delta_4$, we get
$$\sum_{i \in \mathbb{Z}_5} g^{\prime}_{\varepsilon_i \varepsilon_{i+1}}=\frac{10p}{3}+2\sum_{0 \leq r<s<t \leq 4}( \sum_{i \in \mathbb{Z}_5}c^{\varepsilon_i \varepsilon_{i+1}}_{rst}) g^{\prime}_{rst}.$$
Now, it follows from Lemmas \ref{lemma:first} and  \ref{lemma:relgprimeg} that
$$\sum_{i \in \mathbb{Z}_5} \dot g_{\varepsilon_i \varepsilon_{i+1}}=\frac{5p}{3}-\overline{p}+\sum_{0 \leq r<s<t \leq 4}( \sum_{i \in \mathbb{Z}_5}c^{\varepsilon_i \varepsilon_{i+1}}_{rst}) g^{\prime}_{rst}.$$
By substituting the value of $\sum_{i \in \mathbb{Z}_5} \dot g_{\varepsilon_i \varepsilon_{i+1}}$ in Eq. \eqref{eqn:chigamma}, we get
$$\chi_\varepsilon(\Gamma)=\frac{-4p}{3}+\sum_{0 \leq r<s<t \leq 4}( \sum_{i \in \mathbb{Z}_5}c^{\varepsilon_i \varepsilon_{i+1}}_{rst}) g^{\prime}_{rst}.$$
Now, 
\begin{align*}
\rho_\varepsilon(\Gamma)&= 1-\frac{\chi_{\varepsilon}(\Gamma)}{2}-\frac{\lambda_{\varepsilon}}{2}\\
&=1+\frac{2p}{3}-\frac{1}{2}\sum_{0 \leq r<s<t \leq 4}( \sum_{i \in \mathbb{Z}_5}c^{\varepsilon_i \varepsilon_{i+1}}_{rst}) g^{\prime}_{rst} - \frac{\partial g_{\varepsilon_0 \varepsilon_3}}{2}.
\end{align*}
Further from Lemma \ref{lemma:rel2pgamma} we know that
$$2p=6\chi(M)+\sum_{0\leq r<s<t \leq 4} g^{\prime}_{rst}-12h-6$$
Therefore,
$$
\rho_\varepsilon(\Gamma)=-1-4h+2\chi(M)+\frac{1}{3}\sum_{0\leq r<s<t \leq 4} g^{\prime}_{rst}  -\frac{1}{2} \sum_{0 \leq r<s<t \leq 4}( \sum_{i \in \mathbb{Z}_5}c^{\varepsilon_i \varepsilon_{i+1}}_{rst}) g^{\prime}_{rst}-\frac{\partial g_{\varepsilon_0 \varepsilon_3}}{2}.$$

From \cite[page 107]{bb18}, we know that, for a given permutation $\varepsilon=(\varepsilon_0,\dots,\varepsilon_4 =4)$ of $\Delta_4$ and for any $0\leq j<k\leq 4$,
$$\sum_{i \in \mathbb{Z}_5}c^{jk}_{\varepsilon_i \varepsilon_{i+1}\varepsilon_{i+2}}=\frac{2}{3} \mbox{ and } \sum_{i \in \mathbb{Z}_5}c^{jk}_{\varepsilon_i \varepsilon_{i+2}\varepsilon_{i+4}}=-\frac{1}{3}. $$

Thus,$$ \sum_{0 \leq r<s<t \leq 4}( \sum_{i \in \mathbb{Z}_5}c^{\varepsilon_i \varepsilon_{i+1}}_{rst}) g^{\prime}_{rst}=\frac{2}{3} \sum_{i \in \mathbb{Z}_5} g^{\prime}_{\varepsilon_i \varepsilon_{i+1} \varepsilon_{i+2}}-\frac{1}{3} \sum_{i \in \mathbb{Z}_5} g^{\prime}_{\varepsilon_i \varepsilon_{i+2} \varepsilon_{i+4}}.$$
On the other hand,$$
\sum_{0\leq r<s<t \leq 4}g^{\prime}_{rst}= \sum_{i \in \mathbb{Z}_5} g^{\prime}_{\varepsilon_i \varepsilon_{i+1} \varepsilon_{i+2}}+\sum_{i \in \mathbb{Z}_5} g^{\prime}_{\varepsilon_i \varepsilon_{i+2} \varepsilon_{i+4}}.$$
Therefore, 
$$\rho_\varepsilon=-1-4h+2\chi(M)+\frac{1}{2}\sum_{i \in \mathbb{Z}_5}g^{\prime}_{\varepsilon_i \varepsilon_{i+2} \varepsilon_{i+4}}-\frac{\partial g_{\varepsilon_0 \varepsilon_3}}{2}.$$
This proves the result.
\end{proof}

\begin{lemma}\label{lemma:main}
Let $(\Gamma,\gamma)\in \mathbb{G}_4$ be a crystallization of a connected compact PL $4$-manifold $M$ with $h$ boundary components.  Then, for any cyclic permutation $\varepsilon=(\varepsilon_0,\dots,\varepsilon_4 =4)$ of the color set $\Delta_4$,\\
$$\rho_\varepsilon (\Gamma)=-1-4h+2\chi(M)+g_{\varepsilon_0 \varepsilon_1 \varepsilon_3}+g_{\varepsilon_0 \varepsilon_2 \varepsilon_3}+g_{\varepsilon_1 \varepsilon_3 4}+\dot g_{\varepsilon_0 \varepsilon_2 4}+\dot g_{\varepsilon_1 \varepsilon_2 4}.$$
\end{lemma}

\begin{proof}
From Lemma \ref{lemma:rhoepsilon}, we have
\begin{align*}
\rho_\varepsilon(\Gamma)=&-1-4h+2\chi(M)+\frac{1}{2}\sum_{i \in \mathbb{Z}_5}g^{\prime}_{\varepsilon_i \varepsilon_{i+2} \varepsilon_{i+4}}-\frac{\partial g_{\varepsilon_0 \varepsilon_3}}{2}\\
=&-1-4h+2\chi(M)+ g_{\varepsilon_0 \varepsilon_1 \varepsilon_3}+g_{\varepsilon_0 \varepsilon_2 \varepsilon_3}+g_{\varepsilon_0 \varepsilon_2 4}-\frac{1}{2}\partial g_{\varepsilon_0 \varepsilon_2}+g_{\varepsilon_1 \varepsilon_2 4}-\frac{1}{2}\partial g_{\varepsilon_1 \varepsilon_2}\\
&+g_{\varepsilon_1 \varepsilon_3 4}-\frac{1}{2}\partial g_{\varepsilon_1 \varepsilon_3}-\frac{\partial g_{\varepsilon_0 \varepsilon_3}}{2} \mbox{ (by applying Lemmas \ref{lemma:relgprimeg} and \ref{lemma:first}) } \\
=&-1-4h+2\chi(M)+\sum_{i \in \mathbb{Z}_5}g_{\varepsilon_i \varepsilon_{i+2} \varepsilon_{i+4}}-\partial g_{\varepsilon_0 \varepsilon_2}-\partial g_{\varepsilon_1 \varepsilon_2}  \mbox{ (by Proposition \ref{prop:3mfd})} \\
=&-1-4h+2\chi(M)+g_{\varepsilon_0 \varepsilon_1 \varepsilon_3}+g_{\varepsilon_0 \varepsilon_2 \varepsilon_3}+g_{\varepsilon_1 \varepsilon_3 4}+\dot g_{\varepsilon_0 \varepsilon_2 4}+\dot g_{\varepsilon_1 \varepsilon_2 4}\mbox{ (by  Lemma \ref{lemma:first}) }.
\end{align*}
This proves the result.
\end{proof}

Let $(\Gamma,\gamma)$ be a crystallization of a connected compact PL $4$-manifold $M$ with $h$ boundary components.  Thus, $\mathcal{K}(\Gamma)$ has $(1+4h)$ vertices. Let  $(\Gamma',\gamma')$ be the 5-colored graph induced from  $(\Gamma,\gamma)$ which represents $2M$.  Then,  $\mathcal{K}(\Gamma^{\prime})$ has $(2+4h)$ number of vertices. In particular, for each $c\in \{0,1,2,3\}$ there are exactly $h$ vertices labeled by the color $c$  and there are exactly 2 vertices labeled by the color $4$ in $\mathcal{K}(\Gamma^{\prime})$. To get a crystallization $(\bar{\Gamma},\bar \gamma)$  of the closed connected PL 4-manifold $2M$ from  $(\Gamma',\gamma')$ we need to remove $(h-1)$ number of $1$-dipole of each of colors $0,1,2,3$ (i.e., contract $(h-1)$ number of edges of color $c$ which connect two different components of $\Gamma_{\hat c}$ for $c\in \{0,1,2,3\}$) and one $1$-dipole of color 4 (for details on removing $1$-dipole see\cite{fgg86}). Then, $(\bar{\Gamma},\bar \gamma)$ is said to be the induced crystallization of the closed connected PL 4-manifold $2M$ from $(\Gamma,\gamma)$. For $\{i_0,i_1, \dots, i_k\} \subset \Delta_4$, let $\bar g_{i_0i_1 \cdots i_k}$ denote the number of connected components of $\bar \Gamma_{\{i_0,i_1,\dots,i_k\}}$. It is easy to see that $\bar{g}_{ijk}=g^{\prime}_{ijk}-h$ and $\bar{g}_{ij4}=g^{\prime}_{ij4}-2(h-1)$, for $i,j,k \in \{0,1,2,3\}$.

\begin{theorem}\label{theorem:main2}
Let $M$ be a connected compact PL $4$-manifold with $h$ boundary components. Let $m$ and $\overline{m}$ be the rank of the fundamental groups of the manifolds $M$ and $2M$ respectively, where $2M$ is the double of the manifold $M$. Then,
\begin{enumerate}[$(i)$] 
\item $\mathcal{G}(M)\geq 2\chi(M)+2\overline{m}-2$,
\item $\mathcal{G}(M)\geq 2\chi(M)+3m+2h-4$,
\item $\mathcal{G}(M)\geq \mathcal{G}(\partial M)+2\chi(M)+2m+2h-4$.
\end{enumerate} 
\end{theorem}

\begin{proof}
Let $(\Gamma,\gamma)\in \mathbb{G}_4$ be a crystallization of the PL $4$-manifold $M$  and $(\bar{\Gamma},\bar\gamma)$ be the induced crystallization for $2M$ as constructed above. Then, it follows from Proposition \ref{prop:rank} that $\bar m \leq \bar{g}_{ijk} -1$ for $i,j,k \in \Delta_4$.
This implies that $g^{\prime}_{ijk}\geq \bar{m}+h+1$ and $g^{\prime}_{ij4}\geq \bar{m}+2h-1$ for $i,j,k \in \{0,1,2,3\}$. Further, $\dot g_{ij4}\geq h-1$ for $i,j\in\{0,1,2,3\}$.  It follows from Lemma \ref{lemma:rhoepsilon}  that, for any cyclic permutation $\varepsilon=(\varepsilon_0,\dots,\varepsilon_4 =4)$ of the color set $\Delta_4$,
\begin{align*}
\rho_\varepsilon(\Gamma)=&-1-4h+2\chi(M)+\frac{1}{2}\sum_{i \in \mathbb{Z}_5}g^{\prime}_{\varepsilon_i \varepsilon_{i+2} \varepsilon_{i+4}}-\frac{\partial g_{\varepsilon_1 \varepsilon_2}}{2}\\
=&-1-4h+2\chi (M)+\frac{1}{2}(g^{\prime}_{\varepsilon_0 \varepsilon_2 4}+g^{\prime}_{\varepsilon_0 \varepsilon_1 \varepsilon_3}+g^{\prime}_{\varepsilon_0 \varepsilon_2 \varepsilon_3}+g^{\prime}_{\varepsilon_1 \varepsilon_3 4})+\dot{g}_{\varepsilon_1 \varepsilon_2 4}\\
& \mbox{  (by Lemmas \ref{lemma:relgprimeg} and \ref{lemma:first}) }\\
\geq &-1-4h+2\chi (M)+\frac{1}{2}(\bar{m}+2h-1+\bar{m}+h+1+\bar{m}+h+1 \\
& +\bar{m}+2h-1)+(h-1)\\
=&2 \chi (M)+2\bar{m}-2.
\end{align*}
Since this is true for every cyclic permutation $\varepsilon=(\varepsilon_0,\varepsilon_1, \dots, \varepsilon_4 =4)$, we have
$$\rho(\Gamma)=\min\{\rho_ \varepsilon (\Gamma)|\varepsilon \mbox{ is a cyclic permutation of } \Delta_4\}\geq 2 \chi (M)+2\bar{m}-2.$$
Since the crystallization $(\Gamma,\gamma)$ of $M$ is arbitrary,
$$\mathcal{G}(M)=\min\{\rho(\Gamma)|(\Gamma,\gamma) \mbox{ is a crystallization of } M\} \geq 2 \chi (M)+2\bar{m}-2. $$
This proves Part $(i)$.

It follows from Lemma \ref{lemma:main} that, for any cyclic permutation $\varepsilon=(\varepsilon_0,\dots,\varepsilon_4 =4)$ of the color set $\Delta_4$,\\
$$\rho_\varepsilon (\Gamma)=-1-4h+2\chi(M)+g_{\varepsilon_0 \varepsilon_1 \varepsilon_3}+g_{\varepsilon_0 \varepsilon_2 \varepsilon_3}+g_{\varepsilon_1 \varepsilon_3 4}+\dot g_{\varepsilon_0 \varepsilon_2 4}+\dot g_{\varepsilon_1 \varepsilon_2 4}.$$
Because $\Gamma$ is a crystallization of a connected compact PL $4$-manifold $M$ with $h$ boundary components, we have $g(\Gamma_{\{0,1,2,3\}})=1$ and $g(\Gamma_{\Delta_4\setminus\{c\}})=h$, for $0 \leq c \leq 3$. From Proposition \ref{prop:rank}, we know that for $a,b \in \Delta_4$, 
$$m \leq g(\Gamma_{\Delta_4\setminus\{a,b\}})-\big(g(\Gamma_{\Delta_4\setminus\{a\}})+g(\Gamma_{\Delta_4\setminus\{b\}})-1\big).$$ 
This implies,
$$g_{ijk}\geq m+h \mbox{ and } g_{ij4}\geq m+2h-1 \mbox{ for } i,j,k \in \{0,1,2,3\}.$$
Therefore,
$$\rho_\varepsilon(\Gamma)\geq 2\chi(M) -1-4h+2(m+h)+m+2h-1+2(h-1) =2\chi(M)+3m+2h-4.$$
Since this is true for every cyclic permutation $\varepsilon=(\varepsilon_0,\varepsilon_1, \dots, \varepsilon_4 =4)$, we have
$$\rho(\Gamma)=\min\{\rho_ \varepsilon (\Gamma)|\varepsilon \mbox{ is a cyclic permutation of } \Delta_4\}\geq 2\chi(M)+3m+2h-4.$$
Since the crystallization $(\Gamma,\gamma)$ of $M$ is arbitrary,
$$\mathcal{G}(M)=\min\{\rho(\Gamma)|(\Gamma,\gamma) \mbox{ is a crystallization of } M\} \geq 2\chi(M)+3m+2h-4.$$
This proves Part $(ii)$.

It follows from Lemma \ref{lemma:main} that, for any cyclic permutation $\varepsilon=(\varepsilon_0,\dots,\varepsilon_4 =4)$ of the color set $\Delta_4$,\\
$$\rho_\varepsilon (\Gamma)=-1-4h+2\chi(M)+g_{\varepsilon_0 \varepsilon_1 \varepsilon_3}+g_{\varepsilon_0 \varepsilon_2 \varepsilon_3}+g_{\varepsilon_1 \varepsilon_3 4}+\dot g_{\varepsilon_0 \varepsilon_2 4}+\dot g_{\varepsilon_1 \varepsilon_2 4}.$$
Thus, from Lemma \ref{lemma:first} we have
$$\rho_\varepsilon (\Gamma)=-1-4h+2\chi(M)+g_{\varepsilon_0 \varepsilon_1 \varepsilon_3}+g_{\varepsilon_0 \varepsilon_2 \varepsilon_3}+\dot g_{\varepsilon_1 \varepsilon_3 4}+\dot g_{\varepsilon_0 \varepsilon_2 4}+\dot g_{\varepsilon_1 \varepsilon_2 4}+\partial g_{\varepsilon_1 \varepsilon_3}.$$
Since $\partial g_{\varepsilon_1 \varepsilon_3} \geq \mathcal G(\partial M)+h$, $g_{ijk}\geq m+h$ and $\dot g_{ij4}\geq h-1$ for  $i,j,k \in \{0,1,2,3\}$, we have
\begin{align*}
\rho_\varepsilon(\Gamma)& \geq  -1-4h+2\chi(M)+2(m+h)+3(h-1)+\mathcal G(\partial M)+h\\
&=\mathcal G(\partial M)+2\chi(M)+2m+2h-4.
\end{align*}
Since this is true for every cyclic permutation $\varepsilon=(\varepsilon_0,\varepsilon_1, \dots, \varepsilon_4 =4)$, we have
$$\rho(\Gamma)=\min\{\rho_ \varepsilon (\Gamma)|\varepsilon \mbox{ is a cyclic permutation of } \Delta_4\}\geq \mathcal G(\partial M)+2\chi(M)+2m+2h-4$$
Since the crystallization $(\Gamma,\gamma)$ of $M$ is arbitrary,
$$\mathcal{G}(M)=\min\{\rho(\Gamma)|(\Gamma,\gamma) \mbox{ is a crystallization of } M\} \geq \mathcal G(\partial M)+2\chi(M)+2m+2h-4.$$
This proves Part $(iii)$.
\end{proof}

\begin{lemma}\label{lemma:MXI}
Let $(\Gamma,\gamma)$ be a crystallization of a closed connected $3$-manifold $M$ with color set $\Delta_3$. Let $g_{01}$ and $g_{03}$ be the number of components of the graph $\Gamma_{\{0,1\}}$ and $\Gamma_{\{0,3\}}$ respectively. Then $\mathcal G(M\times [0,1])\leq 2(g_{01}+g_{03})-4$.
\end{lemma}

\begin{proof}
Let $2p$ be the number of vertices of $\Gamma$. Let $v_1,v_2,\dots,v_{2p}$ be the vertices of $\Gamma$. Now, choose a fixed order of the colors say, $(0,1,2,3)$. For $0\leq m\leq 4$, let $G_m(i,j,k,l)$ be the graph obtained from $(\Gamma,\gamma)$ by replacing vertices $v_r$ by $v_r^{(m)}$ and by replacing the colors $(0,1,2,3)$ by $(i,j,k,l)$, where $1\leq r \leq 2p$ and $0\leq i \neq j\neq k \neq l\leq 4$. Now consider the five 3-colored graphs $G_0(1,2,3,-)$, $G_1(-,2,3,4)$, $G_2(0,-,3,4)$, $G_3(0,1,-,4)$ and $G_4(0,1,2,-)$, where by `$-$', we mean the corresponding color is missing. Let $(\bar \Gamma,\bar \gamma)$ be a graph obtained by  adding $2p$ edges of color $c$ between $v_r^{(c)}$ and  $v_r^{(c+1)}$ for $0\leq c \leq 3$ and $1\leq r \leq 2p$. Then $(\bar \Gamma,\bar \gamma)$ is a crystallization  of $M\times [0,1]$. Let $\bar g_{ijk}$ be the number of components of the graph $\bar \Gamma_{\{i,j,k\}}$ and $\dot {\bar g}_{ijk}$ be the number of regular components of the graph $\bar \Gamma_{\{i,j,k\}}$  for $0\leq i,j,k\leq 4$. Then it is easy to see that $\bar g_{012}=1+g_{01}$,  $\bar g_{123}=1+g_{12}=1+g_{03}$, $\bar g_{014}=1+g_{01}+g_{03}$, $\dot{\bar g}_{234}=1$ and $\dot{\bar g}_{034}=1$. Here $M\times [0,1]$ is a connected compact $4$-manifold with two boundary components, i.e., $h=2$, and $\chi(M\times [0,1])=0$. It follows from Lemma \ref{lemma:main} that, for $\varepsilon=(2,0,3,1,4)$,
$$\rho_\varepsilon (\bar \Gamma)=-1-4h+2\chi(M\times [0,1])+\bar g_{012}+\bar g_{123}+\bar g_{014}+\dot {\bar g}_{23 4}+\dot {\bar g}_{034}=2(g_{01}+g_{03})-4.$$
Therefore, Then $\mathcal G(M\times [0,1])\leq 2(g_{01}+g_{03})-4$.
\end{proof}

\begin{theorem}
For $M=\mathbb{S}^2\times \mathbb{S}^1$, $\mathbb{RP}^3$ and $\mathbb{S}^1\times \mathbb{S}^1\times \mathbb{S}^1$, the regular genus $\mathcal G(M \times [0,1])=4 \, \mathcal G(M)$ (i.e.,
$\mathcal G(\mathbb{S}^2\times \mathbb{S}^1 \times [0,1]) =\mathcal G(\mathbb{RP}^3 \times [0,1]) =4$ and $\mathcal G(\mathbb{S}^1\times \mathbb{S}^1\times \mathbb{S}^1 \times [0,1]) =12$).
\end{theorem}

\begin{proof}
For $M=\mathbb{S}^2\times \mathbb{S}^1$ and $\mathbb{RP}^3$,  $\mathcal G(M)=1$ and the rank of the fundamental group of $M$ is 1. Further $\mathcal G(\partial (M\times [0,1]))=\mathcal G(M)+\mathcal G(M)=2$, $h=2$ and $m=1$. Thus, by Theorem \ref{theorem:main2}, we have
$$\mathcal G(M \times [0,1]) \geq \mathcal G(\partial (M \times [0,1]))+2\chi(M \times [0,1])+2m+2h-4 = 2+0+2+4-4=4.$$
Further,  there is a crystallization $(\Gamma,\gamma)$ of $M$ such that $g_{ij}=2$ (cf. \cite[Figure 2]{bd14}). Thus, by Lemma \ref{lemma:MXI}, we have $\mathcal G(M\times [0,1])\leq 2(g_{01}+g_{03})-4=4$. Therefore, $\mathcal G(M\times [0,1])=4=4 \, \mathcal G(M)$.

For $M=\mathbb{S}^1\times \mathbb{S}^1\times \mathbb{S}^1$, $\mathcal G(M)=3$ and the rank of the fundamental group of $M$ is 3. Further $\mathcal G(\partial (M\times [0,1]))=\mathcal G(M)+\mathcal G(M)=6$, $h=2$ and $m=3$. Thus, by Theorem \ref{theorem:main2}, we have
$$\mathcal G(M \times [0,1]) \geq \mathcal G(\partial (M \times [0,1]))+2\chi(M \times [0,1])+2m+2h-4 = 12.$$
Further,  there is a crystallization $(\Gamma,\gamma)$ of $M$ such that $g_{01}, g_{03}=4$ (cf. \cite[Figure 5]{bd14}). Thus, by Lemma \ref{lemma:MXI}, we have $\mathcal G(M\times [0,1])\leq 2(g_{01}+g_{03})-4=12$. Therefore, $\mathcal G(M\times [0,1])=12=4 \, \mathcal G(M)$.
\end{proof}

In \cite{bb18}, we defined  weak semi-simple crystallization for closed connected PL 4-manifold. Generalising the definition, here we define weak semi-simple crystallization for a connected compact PL $4$-manifold with $h$ boundary components.

\begin{definition}
 Let $M$ be a connected compact PL $4$-manifold with $h$ boundary components and $m$ be the rank of the fundamental group of $M$. A crystallization $(\Gamma,\gamma)\in \mathbb{G}_4$ of $M$ with color set $\Delta_4$ is said to be a weak semi-simple crystallization of type I if $g_{012}, \, g_{123}=m+h$, $\dot g_{234}, \, \dot g_{034}=h-1$ and $g_{014}=\mathcal G (\partial M)+2h-1$. A crystallization $(\Gamma,\gamma)\in \mathbb{G}_4$ of $M$ with color set $\Delta_4$ is said to be a weak semi-simple crystallization of type II if $g_{012}, \, g_{123}=m+h$, $\dot g_{234}, \, \dot g_{034}=h-1$ and $g_{014}=m+2h-1$.
\end{definition}

Assuming $\varepsilon=(\varepsilon_0,\varepsilon_1,\varepsilon_2, \varepsilon_3, \varepsilon_4=4) = (2,0,3,1,4)$, from the proofs of Part (ii) and Part (iii) of Theorem \ref{theorem:main2}, the following result follows.

\begin{corollary}
 Let $M$ be a connected compact PL $4$-manifold with $h$ boundary components and $m$ be the rank of the fundamental group of $M$. Then 
 \begin{enumerate}[$(i)$] 
 \item  $\mathcal{G}(M)= \mathcal{G}(\partial M)+2\chi(M)+2m+2h-4$ if and only if $M$ admits a weak semi-simple crystallization of type I.
  \item  $\mathcal{G}(M)= 2\chi(M)+3m+2h-4$ if and only if $M$ admits a weak semi-simple crystallization of type II.
 \end{enumerate}
\end{corollary}

Consider the crystallization $(\Gamma,\gamma)$ of $\mathbb{S}^4$ and $v_1,v_2 \in V(\Gamma)$  as in Figure \ref{fig:1}. Then $g_{ijk}=2$ and $g_{ij4}=3$ for $i,j,k\in\{0,1,2,3\}$. For $1\leq i \leq 2$, let $(\Gamma^i,\gamma^i)\in \mathbb{G}_4$ be a crystallization of a connected compact PL $4$-manifold $M_i$ with $h_i$ boundary components. Let $(\Gamma^i,\gamma^i)$ be a weak semi-simple crystallization of type I (resp., of type II) for $1\leq i \leq 2$. Then the graph connected sum $\Gamma^1\#_{uv_1} \Gamma \#_{vv_2} \Gamma^2$ gives a weak semi-simple crystallization of type I (resp., of type II) of $M_1\#M_2$, where $u\in V(\Gamma^1),v\in V(\Gamma^2)$ are interior vertices. Thus, we have the following result.

\begin{theorem}\label{theorem:connected-sum-2}
For $1\leq i \leq 2$, let $M_i$ be a connected compact PL $4$-manifold with $h_i$ boundary components. If the bound obtained in Part $(ii)$ (resp., Part $(iii)$) of Theorem $\ref{theorem:main2}$ is sharp for both $M_1$ and $M_2$ then the bound is sharp for $M_1\# M_2$ as well.
\end{theorem}

\begin{remark}\label{remark:regular-genus}
{\rm
For all the connected compact PL $4$-manifolds constructed in Remark \ref{remark:gem-complexity} and Lemma \ref{lemma:MXI}, our bound for regular genus (in Part (ii) or Part (iii) of Theorem \ref{theorem:main2}) is sharp. Thus, by using Theorem  \ref{theorem:connected-sum-2}, we get a huge class of connected compact PL $4$-manifolds with boundary, for which  our bounds for regular genus is sharp.
}
\end{remark}

\noindent {\bf Acknowledgement:} The authors would like to thank the anonymous referees for many useful comments and suggestions.
The first author is supported by DST INSPIRE Faculty Research Grant (DST/INSPIRE/04/2017/002471).

\end{document}